\documentclass[11pt]{amsart}
\usepackage{amsmath,latexsym,amsfonts,amssymb}
\usepackage{mathrsfs}
\usepackage{fullpage}
\usepackage{enumerate}
\usepackage{amssymb, amsmath, amsfonts}
\usepackage{mathrsfs}
\usepackage{epsfig,amsbsy,amsthm} 
\usepackage{float,epsfig}
\usepackage{fixmath}
\usepackage{graphics}
\usepackage{pgfpages}
\usepackage{caption}
\usepackage{subcaption}
\usepackage{appendix}
\numberwithin{equation}{section}  
\usepackage{hyperref} 
\usepackage{graphicx}
\usepackage{pst-all}
\usepackage{calligra}
\usepackage{color}
\usepackage{tikz}
\DeclareMathAlphabet{\mathpzc}{OT1}{pzc}{m}{it}
\DeclareMathAlphabet{\mathcalligra}{T1}{calligra}{m}{n}
\begin{document}
\newtheorem{theorem}{\bf Theorem}[section]
\newtheorem{proposition}[theorem]{\bf Proposition}
\newtheorem{definition}{\bf Definition}[section]
\newtheorem{corollary}[theorem]{\bf Corollary}
\newtheorem{exam}[theorem]{\bf Example}
\newtheorem{remark}[theorem]{\bf Remark}
\newtheorem{lemma}[theorem]{\bf Lemma}
\newtheorem{assum}[theorem]{\bf Assumption}

\newcommand{\von}{\vskip 1ex}
\newcommand{\vone}{\vskip 2ex}
\newcommand{\vtwo}{\vskip 4ex}
\newcommand{\ds}{\displaystyle}
\def \noin{\noindent}
\newcommand{\be}{\begin{equation}}
\newcommand{\ee}{\end{equation}}
\newcommand{\beno}{\begin{equation*}}
\newcommand{\eeno}{\end{equation*}}
\newcommand{\ba}{\begin{align}}
\newcommand{\ea}{\end{align}}
\newcommand{\bano}{\begin{align*}}
\newcommand{\eano}{\end{align*}}
\newcommand{\bea}{\begin{eqnarray}}
\newcommand{\eea}{\end{eqnarray}}
\newcommand{\beano}{\begin{eqnarray*}}
\newcommand{\eeano}{\end{eqnarray*}}
\def \noin{\noindent}
 \def \tcK{{\tilde {\mathcal K}}}    
\def \O{{\Omega}}
\def \cT{{\mathcal T}}
\def \cV{{\mathcal V}}
\def \cE{{\mathcal E}}
\def \R{{\mathbb R}}
\def \V{{\mathbb V}}
\def \S{{\mathbb S}}
\def \N{{\mathbb N}}
\def \Z{{\mathbb Z}}
\def \Mc{{\mathcal M}}
\def \Cc{{\mathcal C}}
\def \Rc{{\mathcal R}}
\def \Ec{{\mathcal E}}
\def \Gc{{\mathcal G}}
\def \Tc{{\mathcal T}}
\def \Qc{{\mathcal Q}}
\def \Ic{{\mathcal I}}
\def \Pc{{\mathcal P}}
\def \Oc{{\mathcal O}}
\def \Uc{{\mathcal U}}
\def \Yc{{\mathcal Y}}
\def \Ac{{\mathcal A}}
\def \Bc{{\mathcal B}}
\def \k{\mathpzc{k}}
\def \Rp{\mathpzc{R}}
\def \Os{\mathscr{O}}
\def \Js{\mathscr{J}}
\def \Es{\mathscr{E}}
\def \Qs{\mathscr{Q}}
\def \Ss{\mathscr{S}}
\def \Cs{\mathscr{C}}
\def \Ds{\mathscr{D}}
\def \Ms{\mathscr{M}}
\def \Ts{\mathscr{T}}
\def \LL{L^{\infty}(L^{2}(\Omega))}
\def \LH{L^{2}(0,T;H^{1}(\Omega))}
\def \B {\mathrm{BDF}}
\def \el {\mathrm{el}}
\def \re {\mathrm{re}}
\def \e {\mathrm{e}}
\def \div {\mathrm{div}}
\def \CN {\mathrm{CN}}
\def \Rs   {\mathbf{R}_{{\mathrm es}}}
\def \Rb {\mathbf{R}}
\def \Jb {\mathbf{J}}
\def  \apos {\emph{a posteriori~}}

\def\mean#1{\left\{\hskip -5pt\left\{#1\right\}\hskip -5pt\right\}}
\def\jump#1{\left[\hskip -3.5pt\left[#1\right]\hskip -3.5pt\right]}
\def\smean#1{\{\hskip -3pt\{#1\}\hskip -3pt\}}
\def\sjump#1{[\hskip -1.5pt[#1]\hskip -1.5pt]}
\def\jumptwo{\jump{\frac{\p^2 u_h}{\p n^2}}}

\title[]
{Error analysis for parabolic optimal control problems with measure data in a nonconvex polygonal domain}

\author{ Pratibha Shakya}
\address{Department of Mathematics, Indian Institute of Technology Delhi, New Delhi, 110016}

\email{shakya.pratibha10@gmail.com}

\date{}
\maketitle

\begin{abstract}
This paper considers the finite element approximation to parabolic optimal control problems with measure data in a nonconvex polygonal domain. Such problems usually possess low regularity in the state variable due to the presence of measure data and the nonconvex  nature of the domain. The low regularity of the solution  allows the finite element approximations to converge at lower orders. We prove the existence, uniqueness and regularity results for the solution to the control problem satisfying the first order optimality condition. For our error analysis we have used piecewise linear elements for the approximation of the state and co-state variables,  whereas  piecewise constant functions are employed to approximate the control variable. The temporal discretization is  based on the implicit Euler scheme. We derive both a priori and a posteriori error bounds for the state, control  and co-state variables. Numerical experiments are performed to validate the theoretical rates of convergence.
\end{abstract}

\emph{Mathematics subject classification}: 49J20, 49K20, 65N15, 65N30.

\emph{Key words.}
A priori and a posteriori error estimates, finite element method,  measure data, nonconvex polygonal domain, optimal control problem.

\section{Introduction}
The aim of this paper is to study both a priori and a posteriori error analysis  of finite element approximations to the following model control problem:
\begin{align}
\min_{u\in U_{ad}} J(y,u),\label{intro:functional}
\end{align}
where  $J(y,u):=\frac{1}{2}\int_{0}^{T}\|y-y_d\|_{L^2(\Omega)}^2\,dt+\frac{\Lambda}{2}\int_0^T\|u\|_{L^2(\Omega)}\,dt$
with  $u$ represents  the control variable and $y$ indicates the associated state variable. The state equation is given by
\begin{align}
\begin{cases}
\frac{\partial y}{\partial t}-\Delta y=\sigma\tau+u\;\;\;\text{in}\;\Omega_T,\\
y=0\;\;\;\;\text{on}\;\Gamma_T,\\
y(\cdot,0)=y_0\;\;\;\;\text{in}\;\Omega.
\end{cases}\label{intro:state}
\end{align}
\noindent
In the above, $\Omega$ is a  nonconvex polygonal domain in $\mathbb{R}^2$ with Lipschitz boundary $\partial\Omega$. 
Set $\Omega_T=\Omega\times(0,T]$ and  $\Gamma_T=\partial\Omega\times(0,T]$. The boundary $\partial\Omega$ can be expressed as  $\partial\Omega=\displaystyle\cup_{j=1}^m\Gamma_j$ with $\Gamma_j$, $j=1,2,\ldots,m$, are  edges of the polygon.
The constraints on the control variable  are specified through the closed and convex subset of $L^2(0,T;L^2(\Omega))$ as follows:
\begin{align}
U_{ad}:=\{u\in L^2(0,T;L^2(\Omega)): u_a\leq u(x,t)\leq u_b   \; \text{for a. a.}\;(x,t)\in\Omega_T\}.\label{intro:cons}
\end{align} 
Assume that  the given functions  $y_0\in L^2(\Omega),\;y_d\in H^1(0,T;L^2(\Omega)), \;\sigma\in \mathcal{C}([0,T];L^2(\Omega))$ and $\tau\in \mathfrak{B}[0,T]$,  where $\mathfrak{B}[0,T]$ is the space of real and regular  Borel measures in $[0,T]$. 
Further,  the constants $u_a,\,u_b\in\mathbb{R}$ satisfy $u_a<u_b$, the regularization parameter  $\Lambda>0$ and the final time $T<\infty$.

Optimal control problems are widely used  in scientific and engineering applications  \cite{lions71,tiba}. The numerical study of such type of problems began in  early 1970s \cite{Falk,Geveci}.
Thereafter there have been several notable contributions to this discipline and  it is impossible to list all of them. Nevertheless, for the development of the finite element approach for parabolic optimal control problems (POCPs), see  \cite{gohi14,knowles82,vexler11,neitzel12,winther78}
  and references therein.  The authors of \cite{meivex08,meivex081} have  utilized discontinuous Galerkin technique for temporal discretization and established convergence results  for space-time finite element discretizations for POCPs. 
In \cite{vexler11a}, the authors  have  adopted Petrov Galerkin Crank-Nicolson method for discretization of the control problem and derived related error estimates.  The sparse POCPs have been analyzed  by the authors of \cite{19}, where  the control variable is taken to be an element of the  measure space. They have  provided a priori error estimates for the control problem.

Following  the work of  Babu\v{s}ka and Rheinboldt \cite{babuska78}, the adaptive finite element method  has grown popularity in scientific computing. It is well known that a posteriori error estimation is a necessary part of adaptivity for mesh refinement. The pioneer work has been made by  Liu and Yan \cite{liu01} for residual based  a posteriori error estimates,
 Becker {\it et al.} \cite{becker} for dual-weighted goal oriented  adaptivity and Li {\it et al.} \cite{li} for recovery type a posteriori error estimators. 
 A posteriori error analysis for optimal control problems governed by  parabolic equation have been extensively investigated by numerous authors in   \cite{liu03,manohar21,xi,sun,tang}. 

 POCPs are widely encountered in mathematical models representing groundwater contamination transmission,  environmental modeling, petroleum reservoir simulation, and a variety of other  applications.
There are several real-world applications for  POCPs when the state variable possesses less regularity due to the support of the source. 
Essentially, the support for the source function must be relatively tiny in comparison to the real size of the domain $\Omega$. This feature drives us to explore control problems in which the source functions are measure data (elements from $\mathcal{M}(\Omega)$). The POCPs with measure data encounter environmental concerns such as air pollution and waste-water treatment. Due to the presence of measure data, the solution of the state variable possesses less regularity which makes finite element error analysis more challenging. Therefore, an attempt has been made to study the convergence properties of the finite element method for such problems.

The study of optimal control problems governed by partial differential equations over a nonsmooth domain is a difficult task. The existence of re-entrant corners in the domain causes both theoretical and numerical analysis to be complicated. 
However, although there is a significant amount of research on the numerical analysis of the elliptic problem with a nonconvex domain \cite{babuska70n,bacuta01n,bacuta03n,grisvard85n,grisvard92n} and  quite a few works on the parabolic problem    \cite{chatzip06n,chatzip08n}. For optimal control problems, there was not much work done in the nonconvex polygonal domain.  
In recently published article \cite{apel07},  Apel {\it et al.} developed a priori error estimates for the optimal control problem on a nonconvex domain.

The numerical analysis of the problem under consideration is difficult because of the presence of measure data and the nonconvex nature of the domain. The low regularity of the solution  allows the numerical approximations to converge at lower orders. This study aims to look at the finite element approximation and mathematical formulation of the model problem. The results regarding the existence and uniqueness of the solution to the control problem are proved. Based on the necessary optimality condition, the regularity results for the control problem are explored. For the control variable, piecewise constant functions are utilized, whereas piecewise linear and continuous functions are used to approximate the state and co-state variables. The backward-Euler technique is used for temporal discretization. We studied completely discrete finite element approximations of the POCP (\ref{intro:functional})-(\ref{intro:cons}) and established both a priori and a posteriori error bounds for the state, co-state and control variables.

We mention \cite{bocg89,bocg97} for a great introduction to nonlinear parabolic equations with measure data. The author of \cite{cas97} have addressed the semilinear parabolic problems with measure data. Additionally, the asymptotic behavior of  a parabolic equation involving measure data has been  studied  by Gong in \cite{gong12}. For recent research on POCPs with measure data,  we refer to \cite{pshakya19:apos,pshakya19}.

The paper is structured as follows: We introduce some function spaces and preliminary material in Section $2$. We discuss the weak formulation and investigate the existence, uniqueness and regularity results of the  solution to the control problem (\ref{intro:functional})-(\ref{intro:cons}). The convergence analysis  for the a priori error estimates of the space-time finite element approximation to the control problem is discussed in Section 4. In Section 5, we derived a posteriori error estimates for the control problem. In the last section, we perform numerical experiments to  demonstrate the theoretical findings.

  \section{Notation and  wellposedness}
 This section introduces  some function spaces to be used in our analysis.  It also contains  the existence, uniqueness, and regularity  results of the solutions to the POCP (\ref{intro:functional})-(\ref{intro:cons}).
 
 For bounded polygonal domain $\Omega$, let the inner angles of corners of the domain be denoted by 
$ \omega_j$. Set $\beta=\displaystyle \max_{j}\frac{\pi}{\omega_j}\in (\frac{1}{2},1)$. For simplicity, it is assumed that there is only one re-entrant corner with angle $\omega$ such that $\pi<\omega<2\pi$.  For example, the interior angle for  $L$-shape domain $\omega=\frac{3\pi}{2}$, and hence  $\beta=\frac{2}{3}<1$.  Let  $\mathcal{C}(\overline{\Omega})$ denote the space of continuous functions defined on $\overline{\Omega}$.
The space  $W^{m,p}(\Omega)$ indicates  the usual Sobolev spaces \cite{adams} with  norm $\|\cdot\|_{W^{m,p}(\Omega)}$  and semi-norm $|\cdot|_{W^{m,p}(\Omega)}$. Define $W_{0}^{m.p}(\Omega):=\{v\in W^{m,p}(\Omega): v=0\;\;\text{on}\;\;\partial\Omega\}$.  For $p=2$,  the spaces $W^{m,p}(\Omega)$ and  $W^{m,p}_0(\Omega)$ are represented by  $H^{m}(\Omega)$ and $H^m_0(\Omega)$,  respectively  with  norm  $\|\cdot\|_{H^m(\Omega)}$ and semi-norm $|\cdot|_{H^{m}(\Omega)}$. 
In particular, for $0<s<1$ and $1 < p\leq \infty$, the  norm on the fractional order Sobolev space $W^{s,p}(\Omega)$ is given by
\begin{align*}
\|v\|_{W^{s,p}(\Omega)}=\Big(\|v\|_{L^p(\Omega)}^p+\int_{\Omega}\int_{\Omega}\frac{|v(x)-v(y)|^{p}}{|x-y|^{2+ps}}\,dxdy\Big)^{\frac{1}{p}}.
\end{align*}

\noindent
Set  $H^{m,m'}(\Omega_T)=L^2(0,T;H^m(\Omega))\cap H^{m'}(0,T; L^2(\Omega))$  with the standard norm
\begin{align*}
\|w\|_{H^{m,m'}(\Omega_T)}:=\Big(\int_0^T\|w(\cdot,t)\|^2_{H^m(\Omega)}\,dt+\int_{\Omega}\|w(x,\cdot)\|^2_{H^{m'}([0,T])}\,dx\Big)^{\frac{1}{2}},
\end{align*}
 where $\|\cdot\|_{H^{m'}([0,T])}$ denote the norm on $H^{m'}([0,T])$.\\
 
\noindent 
Further, 
let $\mathcal{X}(0,T), \;\mathcal{\hat{X}}(0,T)$ and  $W(0,T)$ denote $L^2(0,T;H^1_0(\Omega))\cap H^1(0,T;H^{-1}(\Omega))$, \\$  L^2(0,T;H^{1+s}(\Omega)\cap H^1_0(\Omega))\cap H^1(0,T;L^2(\Omega))$ and $L^2(0,T;H^1_0(\Omega))\cap L^{\infty}(0,T;L^2(\Omega))$, respectively for $s\in(\frac{1}{2},\beta)$. The symbols $ (\cdot,\cdot)$  and $(\cdot, \cdot)_{\Omega_T}$ denote the $L^2$-inner product on $L^2(\Omega)$  and $L^2(0,T;L^2(\Omega))$, respectively.
Hereafter  $C$ denotes   a positive generic constant which is  independent of the mesh parameters $h$ and $k$, which may depend on final time $T$. 
 
We employ the transposition approach developed by Lions and Magenes (cf. \cite{lionsm72})  to assertion that the state equation (\ref{intro:state}) has a unique solution. 
The  weak form of \eqref{intro:state} is stated as: Find $y\in W(0,T)$ such that 
  \begin{equation}
-(y,\frac{\partial w}{\partial t})_{\Omega_T}+(\nabla y,\nabla w)_{\Omega_T}=\langle\sigma\tau,w\rangle_{\Omega_T}+(u,w)_{\Omega_T}+(y_0,w(\cdot,0))\;\;\forall w\in \mathcal{X}(0,T), \label{state:control:eq:1}
\end{equation}
where we utilized  $w(\cdot,T)=0$ and  $\langle \sigma\tau,w\rangle_{\Omega_T}$ is defined as
\begin{equation*}
\langle\sigma\tau,w\rangle_{\Omega_T}=\int_0^T\int_{\Omega}\sigma(x,t)w(x,t)\,dx\,d\tau(t)\;\;\;\;\forall w\in\mathcal{C}([0,T];L^2(\Omega)).
\end{equation*}
In the subsequent theorem, we provide a priori bounds for the state variable which are essential to our analysis. For a  proof,  see \cite{pshakya19}.
\begin{theorem}\label{stability}
For  $u\in L^2(0,T;L^2(\Omega))\cap L^{\infty}(0,T;L^2(\Omega))$, assume that the given functions $\sigma\in \mathcal{C}([0,T];L^2(\Omega))$, $\tau\in \mathfrak{B}[0,T]$ and $y_0\in L^2(\Omega)$. Then,  the unique solution $y\in W(0,T)$ of the problem (\ref{intro:state}) 
exists and satisfies a priori bound:
\begin{align*}
\|y\|_{L^2(0,T;H^1_0(\Omega))}+\|y\|_{L^{\infty}(0,T;L^2(\Omega))}\leq C\Big(\|\sigma\|_{L^{\infty}(0,T;L^2(\Omega))}
\|\tau\|_{\mathfrak{B}[0,T]}+\|u\|_{L^2(0,T;L^2(\Omega))}+\|y_0\|_{L^2(\Omega)}\Big).
\end{align*}
\end{theorem}
\noindent
The weak formulation of the 
 control problem (\ref{intro:functional})-(\ref{intro:cons}) is as follows:
\begin{align}
\begin{cases}
 \displaystyle\min_{u\in U_{ad}} J(y,u)\\
 -(y,\frac{\partial w}{\partial t})_{\Omega_T}+(\nabla y,\nabla w)_{\Omega_T}=\langle\sigma\tau,w\rangle_{\Omega_T}+(u,w)_{\Omega_T}+(y_0,w(\cdot,0))\;\;\forall w\in \mathcal{X}(0,T),
\end{cases}\label{funct:weak}
\end{align}
where  $w(\cdot,T)=0$ and $\langle\sigma\tau,w\rangle_{\Omega_T}$  is defined as  before.

  From the standard arguments, there exists a unique solution $(y,u)$ for the problem (\ref{funct:weak}). Let $\mathcal{J}(u):=J(y(u),u)$ denote the reduced cost functional, where for each $u\in L^2(0,T;L^2(\Omega))$ the state $y(u)$ is the  weak solution of (\ref{state:control:eq:1}). 
 It should be noted that the cost functional of the  optimal control problem (\ref{funct:weak}) is strictly convex and hence, in light of the Theorem \ref{stability}, it  is bounded. This ensures the existence of  an optimal solution. 
 Since $\mathcal{J}$ is  twice Fr\'{e}chet differentiable convex function and $U_{ad}$ is a closed convex subset of $L^2(0,T;L^2(\Omega))$, the existence of unique control is guaranteed.  For further reading, we refer to  \cite{malanowski82}.

We now  state  the first-order optimality condition which is necessary and sufficient for the optimal control problem (\ref{funct:weak}).

\begin{lemma}
The optimal control problem (\ref{funct:weak}) has  a unique solution $(y,u)$. Then  there exists a  co-state  variable  $\phi\in \hat{\mathcal{X}}(0,T)$ which is the solution of  
\begin{align}
\begin{cases}
-(\frac{\partial \phi}{\partial t},w)+(\nabla \phi,\nabla w)=(y-y_d,w)\;\;\forall w\in L^2(0,T;H^1_0(\Omega)),\\
\phi(\cdot,T)=0\;\;\text{in}\;\;\;\Omega.\label{adjoint:eq:in:time}
\end{cases}
\end{align}
Moreover, the following variational inequality is satisfied:
\begin{equation}
\mathcal{J}'(u)(\hat{u}-{u})=\int_{0}^T(\Lambda {u}+\phi,\hat{u}-u)\,dt\geq 0\;\;\;\;\forall \hat{u}\in U_{ad}.\label{optimality:condition:exp}
\end{equation} 
\end{lemma}
\begin{proof}
Let $\hat{u}\in U_{ad}$ be arbitrary and let $u$ be the optimal solution.  Since $U_{ad}$ is convex, for $\lambda\in(0,1]$, we have  $(u+\lambda(\hat{u}-u))\in U_{ad}$.
Note that, $u$ is optimal, this implies $\mathcal{J}(u+\lambda(\hat{u}-u))\geq \mathcal{J}(u)$ and hence
\begin{align*}
\frac{1}{\lambda}(\mathcal{J}(u+\lambda(\hat{u}-u))- \mathcal{J}(u))\geq 0\;\text{for}\; \lambda\in(0,1].
\end{align*}
Letting $\lambda\rightarrow 0$, we get $\mathcal{J}'(u)(\hat{u}-u)\geq 0$, which validates 
(\ref{optimality:condition:exp}).
\end{proof}
\noindent
It is easy to verify that  the variational inequality (\ref{optimality:condition:exp}) implies
\begin{equation}
 u=P_{[u_a,u_b]}\left(-\frac{\phi}{\Lambda}\right),\label{exp:for:qbar:time}
 \end{equation}
  where  $P_{[u_a,u_b]}$ indicates the point-wise projection on  $U_{ad}$, and is defined by
\begin{equation}
P_{[u_a,u_b]}(\tilde{u}(x,t)):=\min(u_b,\,max(u_a,\, \tilde{u}(x,t))).\label{pointwise:projection}
\end{equation}

\noindent
Moreover,  there exists a  positive constant $\gamma$ such that  the following
\begin{equation}
\mathcal{J}^{\prime \prime}(u)(\tilde{u},\tilde{u})\geq \gamma \|\tilde{u}\|^{2}_{L^2(0,T;L^2(\Omega))}\;\;\;\forall \tilde{u}\in L^{2}(0,T;L^2(\Omega))\label{opsc1}
\end{equation} 
holds.

\noindent
We now state the  regularity results associated with the backward parabolic problem   without proof. The proof of which can be found in   \cite{chatzip06n}.
 \begin{proposition}\label{psi:stability}
For $g\in L^2(0,T;L^2(\Omega))$, let $\eta\in \hat{\mathcal{X}}(0,T)$ be the  solution of 
\begin{align}
 \begin{cases}
-\frac{\partial\eta}{\partial t}-\Delta\eta=g\;\;\;\;\;&\text{in}\;\;\Omega\times [0,T),\\
\eta=0\;\;\;\;\;\;\;\;\;&\text{on}\;\;\partial\Omega\times[0,T),\\
\eta(\cdot,T)=0\;\;\;\;\;&\text{in}\;\;\Omega. \label{adjoint}
\end{cases}
\end{align} 
Thus, we have the following  a priori bounds:
\begin{align*}
\|\eta\|_{H^1(0,T;L^2(\Omega))}+
\|\eta\|_{L^2(0,T;H^{1+s}(\Omega))}&\leq C_R\|g\|_{L^2(0,T;L^2(\Omega))},\\
\|\eta(\cdot,0)\|_{H^1(\Omega)}&\leq C_R\|g\|_{L^2(0,T;L^2(\Omega))},
\end{align*}
where $C_R$ is a positive regularity constant.
\end{proposition}
\noindent
 Now, we discuss the regularity of the solution  to the problems (\ref{funct:weak}) and  (\ref{adjoint:eq:in:time}) in the following lemma. 
 \begin{lemma}\label{regularity:u}
 Let $(y,u)$ be the solution of the optimization problem (\ref{funct:weak}), and let $\phi$ be the solution of (\ref{adjoint:eq:in:time}). Then,  we have 
 \begin{align*}
 (y,u,\phi)\in  W(0,T)\times \hat{\mathcal{X}}(0,T)\times \hat{\mathcal{X}}(0,T).
 \end{align*}
 \end{lemma}
 \begin{proof}
We deduce from Theorem \ref{stability}  that $y\in W(0,T)$. For $y_d\in L^2(0,T;L^2(\Omega))$ implies $\phi\in \hat{\mathcal{X}}(0,T)$, which together with (\ref{exp:for:qbar:time}) gives $u\in \hat{\mathcal{X}}(0,T)$. 
 \end{proof}

 \section{Finite Element Discretization}

This section is focused on the approximation of the POCP (\ref{funct:weak}) using the finite element technique.

Let $h=\displaystyle\max_{K\in\mathcal{T}_h} diam(K)$ be the maximum diameter of the triangles formed by the quasi-uniform triangulation $\mathcal{T}_h$ of $\Omega$. Let $\mathcal{E}_h$ indicate the set of all interior edges. For a piecewise scalar function $v$, the jump of $v$ across an edge $e$ is given by $\sjump{v}=v|_{K^+}-v|_{K^{-}}$, where $K^{+}$ and $K^{-}$ are two triangles that share the common edge $e$.

 \noindent
The finite element spaces defined for a particular triangulation $\mathcal{T}_h$ are given as: 
\begin{align*}
\mathbb{W}_{h}&:=\{w_h\in\mathcal{C}(\overline{\Omega}):\,w_h|_{K} \;\;\text{is a linear polynomial}\},\\
\mathbb{U}_{ad,h}&:=\{\hat{u}_h\in U_{ad}:\, \hat{u}_h|_{K}\;\;\text{is a constant}\}.
\end{align*}  
 With $\mathbb{W}_h$ defined as above, set $\mathbb{W}_h^0=\mathbb{W}_h\cap H_0^1(\Omega)$.
 The following inverse estimate holds  for $w_h\in\mathbb{W}_h$  (cf. \cite{ciar78}):
\begin{align}
\|w_h\|_{H^{p_2}(\Omega)}\leq Ch^{p_1-p_2}\|w_h\|_{H^{p_1}(\Omega)},\;\;0\leq p_1\leq p_2\leq 1,\;\;\;\forall w_h\in\mathbb{W}_h.\label{inverse:estimate:1}
\end{align}
In the following lemmas, we recall the approximation properties associated with the elliptic projection and the $L^2$-projection  (cf. \cite{chatzip06n,chrysafinos02}).

\begin{lemma}\label{projection:rh}
The elliptic projection $\mathcal{P}_h^1:H_0^1(\Omega)\rightarrow \mathbb{W}_h^0$ is defined as
\begin{equation*}
(\nabla(\mathcal{P}_h^1 w-w),\nabla w_h)=0\;\;\;\;\;\forall w_h\in \mathbb{W}_h^0.\label{rh:projection:operator}
\end{equation*}
 Then, for $s\in (\frac{1}{2},\beta)$, we have  
\begin{align*}
\|w-\mathcal{P}_h^1 w\|_{L^2(\Omega)}&+h^{s}\|\nabla w-\nabla\mathcal{P}_h^1 w\|_{L^2(\Omega)}\leq Ch^{2s}\|w\|_{H^{1+s}(\Omega)}.
\end{align*}
Moreover,
\begin{align*}
\|w-\mathcal{P}_h^1 w\|_{L^2(\Omega)}\leq Ch^{s}\|w\|_{H^1(\Omega)}.
\end{align*}
\end{lemma}

\begin{lemma}\label{projection:lh}
The $L^2$-projection  $\mathcal{P}_h^0:L^2(\Omega)\rightarrow \mathbb{W}_h $ is defined as
\begin{equation}
(\mathcal{P}_h^0 w-w,w_h)=0\;\;\;\;\;\forall w_h\in \mathbb{W}_h.\label{l2:projection:operator}
\end{equation}
 Then, for $s\in (\frac{1}{2},\beta)$, we have the following estimates
\begin{align*}
&\|w-\mathcal{P}_h^0 w\|_{H^{-1}(\Omega)}+h^{s}\|w-\mathcal{P}_h^0 w\|_{L^2(\Omega)}\leq Ch^{2s}\|w\|_{H^{1}(\Omega)},\\
&\|w-\mathcal{P}_h^0 w\|_{H^1(\Omega)}\leq  Ch^{2s-1}\| w\|_{H^{1+s}(\Omega)}.
\end{align*}
\end{lemma}

\noindent 
With $y_{h,0}= \mathcal{P}_h^0 y_0$,
 the spatially discrete  approximation of the  problem  (\ref{funct:weak}) is to find $(y_h(t),u_h(t))\in  L^2(0,T;\mathbb{W}_h^0)\times L^2(0,T;\mathbb{U}_{ad,h})$ such that
\begin{align}
\min_{u_h\in L^2(0,T;\mathbb{U}_{ad,h})} J_h(y_h,u_h)= \frac{1}{2}\int_{0}^{T}\left\{\|y_h-y_{d}\|^2_{L^2(\Omega)}+\Lambda\|u_h\|^2_{L^2(\Omega)}\right\}\,dt\label{semi:discrete:functional}
\end{align}
subject to the state equation
\begin{equation}
-(y_h,\frac{\partial w_h}{\partial t})_{\Omega_T}+(\nabla y_h,\nabla w_h)_{\Omega_T}=\langle\sigma\tau,w_h\rangle_{\Omega_T}+(u_h,w_h)_{\Omega_T}+(y_{h,0},w_h(\cdot,0)),\label{semi:discrete:state:eq}
\end{equation}
where $w_h(\cdot,T)=0$ and 
 \begin{align*}
 \langle\sigma\tau,w_h\rangle_{\Omega_T}=\int_0^T\int_{\Omega}\sigma w_h\,dx\,d\tau(t)\;\;\;\forall w_h\in H^1(0,T;\mathbb{W}_h^0).
\end{align*}

Next, we consider the completely discrete approximation  of the  spatially discrete problem (\ref{semi:discrete:functional})-(\ref{semi:discrete:state:eq}). For this,
we introduce a partition of  $[0,T]$  as $0=t_0<t_1<\ldots<t_{N-1}<t_{N}=T$. Utilizing the time partition, we notice that the time  interval $[0,T]$ is divided  into subintervals $I_{n}=(t_{n-1},t_n]$ with time  step $k_n=t_n-t_{n-1}$ and $k=\displaystyle\max_{1\leq n\leq N}k_n$. We assume that the time partition is quasi-uniform, i.e., there exist positive constants $c_1$ and $c_2$  such that $c_1k_n\leq k\leq c_2k_n$ holds for each $n\in[1: N]$. Set $\chi^n:=\chi(x,t_n)$  for any sequence of functions $\{\chi^n\}_{n=0}^{N}$ defined in $\Omega$, and define $D_{k_n}\chi^{n+1}=\frac{(\chi^{n+1}-\chi^{n})}{k_n}$.
Construct 
the finite element space $\mathbb{W}_h^n\subset H^1_0(\Omega)$ related with the mesh $\mathcal{T}_h^n$.  Similar to $\mathcal{E}_h$, we indicate $\mathcal{E}_h^n$ as set of all internal edges of $\mathcal{T}_h^n$. 
For $ n\in[1:N]$, define the discrete  space for the control variable as:  
\begin{align*}
\mathbb{U}^n_{ad}:=\{\tilde{u}\in \mathbb{U}_{ad}: \;\,\tilde{u}|_{I_n\times K}=constant,\;\;K\in \mathcal{T}_h^n\}. 
\end{align*}

\noindent
Let $V_{k}$ indicate the space of piecewise constant functions on the time partition. Define $P_k^n:L^2(0,T)\rightarrow I_n$ as
 \begin{align*}
  P_{k}^n v:=(P_{k}v)(t)|_{I_n}=\frac{1}{k_n}\int_{I_n}v(t)\,dt\;\;\text{for}\;\;t\in I_n,
  \end{align*}
and indicate $P_{k}:L^2(0,T)\rightarrow V_k$ such that $P_k v|_{I_n}=P_k^nv$.
Then, $P_k$ fulfils
  \begin{align}
  \|(I-P_{k})v\|_{L^2(0,T;L^2(\Omega))}\leq Ck\|v_t\|_{L^2(0,T;L^2(\Omega))} \;\;\forall v\in H^1(0,T;L^2(\Omega)).\label{time:pro}
  \end{align}
  
\noindent  
 The completely discrete approximation of (\ref{semi:discrete:functional})-(\ref{semi:discrete:state:eq})  is defined as: Find $(y^{n}_h,u^{n}_{h})\in  \mathbb{W}_h^n\times \mathbb{U}_{ad}^n$  for $n\in[1:N]$, such that
\begin{align}
\min_{u^{n}_{h}\in \mathbb{U}^n_{ad}} \mathcal{J}_n(u_h^n):=J(y_h^n,u_h^n)=
\frac{1}{2} \sum_{n=1}^{N}\int_{t_{n-1}}^{t_n}\Big\{\|y^{n}_{h}-P_{k}^ny_{d}\|^{2}_{L^{2}(\Omega)}+\Lambda\|u^{n}_{h}\|^{2}_{L^{2}(\Omega)}\Big\}dt\label{fully:discrete:functional}
\end{align}
subject to
\begin{align}
\begin{cases}
(D_{k_n}y_h^n,w_h)+(\nabla y_h^n,\nabla w_h)=\langle\sigma\tau,w_h\rangle_{I_n}+(u_h^n,w_h)\;\;\forall w_h\in \mathbb{W}_h^n,\\
y_h^0=y_{h,0},\label{fully:discrete:state}
\end{cases}
\end{align}
where  $\langle\sigma\tau,w_h\rangle_{I_n}$ is given by
\begin{equation*}
\langle\sigma\tau,w_h\rangle_{I_n}=\frac{1}{k_n}\int_{t_{n-1}}^{t_{n}}\int_{\Omega}\sigma(x,t)w_h(x)\,dx\,d\tau(t)\;\;\forall w_h\in \mathbb{W}_h^n.
\end{equation*}

In the following, we need to investigate the stability  behaviour of the solution to the completely discrete state equation (\ref{fully:discrete:state}) concerning the initial value $y_0$, the measure data $\sigma\tau$ and the discrete control variable $u_h^n$. 
 \begin{lemma}\label{fully:discrete:inter:lemma:1}
   For $n\in [1:N]$, consider $y_{h,0}=\mathcal{P}_h^0 y_0$ and let $y_h^n\in \mathbb{W}_h^n$ be the solution of (\ref{fully:discrete:state}). Then, we have the following estimates:
  \begin{align*}
  \sum_{n=1}^{N}\|y_h^n-y_h^{n-1}\|^2_{L^2(\Omega)}+Ck\|y_h^N\|_{H^1(\Omega)}^2&\leq C\Big(\|\sigma\|^2_{L^{\infty}(0,T;L^2(\Omega))}\|\tau\|^2_{\mathfrak{B}[0,T]}+kh^{-2}\|y_0\|^2_{L^2(\Omega)}\Big)\nonumber\\&\quad\quad+Ck\|u_h^n\|^2_{L^2(0,T;L^2(\Omega))},
  \end{align*}
  and
  \begin{align*}
 \|y_h^N\|^2_{L^2(\Omega)}+C\sum_{n=1}^{N}k\|y_h^n\|_{H^1(\Omega)}^{2}&\leq C\Big(\|\sigma\|^2_{L^{\infty}(0,T;L^2(\Omega))}\|\tau\|^2_{\mathfrak{B}[0,T]}+\|y_0\|^2_{L^2(\Omega)}\Big)\nonumber\\&\quad\quad+Ck\|u_h^n\|^2_{L^2(0,T;L^2(\Omega))}.
\end{align*}
\end{lemma}
\begin{proof}
The proof proceeds in the same lines as  \cite{pshakya19}, hence we omit the details.
\end{proof}

The completely  discrete POCP (\ref{fully:discrete:functional})-(\ref{fully:discrete:state}) has a unique solution $(y^{n}_{h},u^{n}_{h})$ for $n\in [1:N]$, such that  the triplet $(y^{n}_{h},u^{n}_{h},\phi^{n-1}_{h})$   fulfils
\begin{align}
(D_{k_n}y_h^n,w_h)+(\nabla y_h^n,\nabla w_h)&=\langle\sigma\tau,w_h\rangle_{I_n}+(u_h^n,w_h)\;\;\forall w_h\in \mathbb{W}_h^n,\label{fully:optimal:state}\\
-(D_{k_n}\phi_h^n,w_h)+(\nabla \phi_h^{n-1},\nabla w_h)&=(y_h^n-P_k^ny_d,w_h)\;\;\forall w_h\in \mathbb{W}_h^n,\label{fully:discrete:costate}\\
\phi^N_h&=0,\label{fully:final:time}\\
(\Lambda u_h^n+\phi_h^{n-1},\hat{u}_h^n-u_h^n)&\geq 0\;\;\;\;\;\forall \hat{u}_h^n\in \mathbb{U}_{ad}^n.\label{fully:discrete:optimality:condition}
\end{align}

\section{A priori error estimates}
This section concerns a priori error estimates for the control, state and co-state variables.

For $n\in[1:N]$, on each time interval $I_n$,  define $Y_h(t):=y_h^n$, $U_h(t):=u_h^n$, and continuous piecewise linear interpolant $\Phi_h(t)$ as
\begin{align*}
\Phi_h(t):=\frac{(t_n-t)}{k_n}\phi_h^{n-1}+\frac{(t-t_{n-1})}{k_n}\phi_h^n.\end{align*}

 \noindent
Here, we first  introduce the  auxiliary problems for the state and co-state variables as follows: Find $y_h^n(u)\in \mathbb{W}_h^n$ such that 
\begin{align}
\begin{cases}
(D_{k_n} y_h^n(u),w_h)+(\nabla y_h^n(u),\nabla w_h)=\langle\sigma\tau,w_h\rangle_{I_n}+(u,w_h)\;\;\;\;\;\forall w_h\in \mathbb{W}_h^n,\;n\geq 1,\\y_h^0(u)=y_{h,0},\label{inter:state:measure:data:in:time}
\end{cases}
\end{align} 
and  for $n<N$, let $\phi_h^{n-1}(u)\in \mathbb{W}_h^n$ be the solution of 
\begin{align}
\begin{cases}
-(D_{k_n} \phi_h^n(u),w_h)+(\nabla \phi_h^{n-1}(u),\nabla w_h)=(y_h^n(u)-P_k^ny_{d},w_h)\;\;\;\forall w_h\in \mathbb{W}_h^n,\\
\phi^N_h(u)=0.\label{inter:fully:discrete:costate}
\end{cases}
\end{align}
The preliminary error bounds for the state and co-state variables  are provided in the next lemma.
\begin{lemma}\label{inter:error:estimate:theorem:for:state:time}
Under the assumption of Theorem \ref{stability}, let
$(y,\phi)$ and $(Y_h(u),\Phi_h(u))$ be the solutions  (\ref{state:control:eq:1}), (\ref{adjoint:eq:in:time}) and   (\ref{inter:state:measure:data:in:time})-(\ref{inter:fully:discrete:costate}), respectively. Then,  for $y_d\in H^1(0,T;L^2(\Omega))$ and $s\in (\frac{1}{2},\beta)$, the following estimates  
\begin{align}
\|y-{Y}_h(u)\|_{L^2(0,T;L^2(\Omega))}&\leq C(h^{2s}k^{-\frac{1}{2}}+k^{\frac{1}{2}}+h^{s})\Big\{\|y_0\|_{L^2(\Omega)}+\|\sigma\|_{L^\infty(0,T;L^2(\Omega))}\|\tau\|_{\mathfrak{B}[0,T]}\nonumber\\&\quad\quad+\|u\|_{L^2(0,T;L^2(\Omega))}\Big\},\label{fuu:es1}\\
\|\phi-\Phi_h(u)\|_{L^2(0,T;L^2(\Omega))}&\leq C(h^{2s}k^{-\frac{1}{2}}+k^{\frac{1}{2}}+h^s)\Big\{\|y_h^n(u)\|_{L^2(0,T;L^2(\Omega))}+\|P_k^ny_d\|_{L^2(0,T;L^2(\Omega))}\Big\}\nonumber\\&\quad\quad+Ck\|y_{d,t}\|_{L^2(0,T;L^2(\Omega))}+C\|y-{Y}_h(u)\|_{L^2(0,T;L^2(\Omega))}\label{fuu:es2}
\end{align}
hold.
\end{lemma}
\begin{proof}
 Let $\eta$  solves the problem (\ref{adjoint}) with $g\in L^2(0,T;L^2(\Omega))$.  In analogy with (\ref{state:control:eq:1}) and (\ref{inter:state:measure:data:in:time}), we write
 \begin{align}
\int_{\Omega_T}(y-{Y}_h(u))g\,dxdt&=\int_0^T\int_{\Omega}(y-{Y}_h(u))(-\frac{\partial \eta}{\partial t}-\Delta\eta)\,dxdt\nonumber\\
&=-(y, \frac{\partial \eta}{\partial t})_{\Omega_T}+(\nabla y,\nabla \eta)_{\Omega_T}+\sum_{n=1}^{N}\int_{t_{n-1}}^{t_n}\Big((y_h^n(u), \frac{\partial \eta}{\partial t})-(\nabla y_h^n(u),\nabla\eta)\Big)\,dt\nonumber
\\
&=\langle\sigma\tau,\eta\rangle_{\Omega_T}+(u,\eta)_{\Omega_T}+(y_0,\eta(\cdot,0))\nonumber\\&\quad+\sum_{n=1}^{N}\int_{t_{n-1}}^{t_n}\Big\{k_{n}^{-1}(y_h^n(u),\eta^n-\eta^{n-1})-(\nabla y_h^n(u),\nabla \eta)\Big\}\,dt.\nonumber
\end{align}
Use of summation by parts and $\eta^N=0$ gives
\begin{align}
\int_{\Omega_T}(y-{Y}_h(u))g\,dxdt&=\langle\sigma\tau,\eta\rangle_{\Omega_T}+(y_0,\eta(\cdot,0))-\sum_{n=1}^{N}\int_{t_{n-1}}^{t_n}\Big\{k_{n}^{-1}(y_h^n(u)-y_h^{n-1}(u),\eta^{n-1})\nonumber\\&\quad+(\nabla y_h^n(u),\nabla \eta)\Big\}\,dt+(y_h^N(u),\eta^N)-(y_{h,0},\eta(\cdot,0))+(u,\eta)_{\Omega_T}\nonumber\\
&=-\sum_{n=1}^{N}\int_{t_{n-1}}^{t_n}\Big\{k_{n}^{-1}(y_h^n(u)-y_h^{n-1}(u),\eta^{n-1})+(\nabla y_h^n(u),\nabla \eta)\Big\}\,dt+\langle\sigma\tau,\eta\rangle_{\Omega_T}\nonumber\\&\quad+(y_0-y_{h,0},\eta(\cdot,0))+(u,\eta)_{\Omega_T}.\label{f:d:e:1}
\end{align}
Utilize (\ref{inter:state:measure:data:in:time}) to have
\begin{align}
\sum_{n=1}^{N}\{(D_{k_n} y_h^n(u),P_k^n\mathcal{P}_h^1\eta)+(\nabla y_h^n(u),\nabla P_k^n\mathcal{P}_h^1\eta)\}=\sum_{n=1}^{N}\langle\sigma\tau,P_k^n\mathcal{P}_h^1\eta\rangle_{I_n}+\sum_{n=1}^{N}(u,P_k^n\mathcal{P}_h^1\eta).\label{main:state:error:eq:1}
\end{align}
Applications of 
 (\ref{f:d:e:1}) and (\ref{main:state:error:eq:1}) together with the fact  $\int_{t_{n-1}}^{t_n}(\eta-P_k^n\eta)\,dt=0$ lead to
\begin{align}
&\int_{\Omega_T}(y-{Y}_h(u))g\,dxdt=-\sum_{n=1}^{N}\int_{t_{n-1}}^{t_n}\Big\{k_{n}^{-1}(y_h^n(u)-y_h^{n-1}(u),\eta^{n-1}-P_k^n\mathcal{P}_h^1\eta)\nonumber\\&+(\nabla y_h^n(u),\nabla (P_k^n\eta-P_k^n\mathcal{P}_h^1\eta))\Big\}\,dt+\Big\{\langle\sigma\tau,\eta\rangle_{\Omega_T}-\sum_{n=1}^{N}\int_{t_{n-1}}^{t_n}\langle\sigma\tau,P_k^n\mathcal{P}_{h}^1\eta\rangle_{I_n} dt\Big\}\nonumber\\&+\Big\{(y_0-y_{h,0},\eta(\cdot,0))\Big\}+\Big\{(u,\eta)_{\Omega_T}-\sum_{n=1}^N\int_{t_{n-1}}^{t_n}(u,P_k^n\mathcal{P}_h^1\eta)\,dt\Big\}\nonumber\\=:&I_1+I_2+I_3+I_4.\label{rearranging:fully:discrete:1}
\end{align}
For $I_1$, using  the definition of  elliptic projection  and the fact $y_h^n(u)\in \mathbb{W}_h^n$,  we obtain 
\begin{equation*}
\int_{t_{n-1}}^{t_n}(\nabla y_h^n(u),\nabla(P_k^n\eta-P_k^n\mathcal{P}_h^1\eta))\,dt=0.\label{psi:diff:psi:bar}
\end{equation*}
 Apply the Cauchy-Schwarz inequality  to have 
\begin{align*}
|I_1|&=\Big|-\sum_{n=1}^N\int_{t_{n-1}}^{t_n}\Big\{{k_{n}}^{-1}(y_h^n(u)-y_h^{n-1}(u),\eta^{n-1}-P_k^n\mathcal{P}_h^1\eta)\Big\}\,dt\Big|\leq F_1\cdot F_2,
\end{align*}
where 
\begin{equation*}
F_1=\left(\sum_{n=1}^{N}\|y_h^n(u)-y_h^{n-1}(u)\|^2_{L^2(\Omega)}\right)^{\frac{1}{2}},\;\;\;\;
F_2=\left(\sum_{n=1}^{N}\|\eta^{n-1}-P_k^n\mathcal{P}_h^1\eta\|^2_{L^2(\Omega)}\right)^{\frac{1}{2}}.
\end{equation*}
 An application of Lemma \ref{fully:discrete:inter:lemma:1} yields
\begin{equation*}
F_1\leq C\Big(k^{\frac{1}{2}}h^{-1}\|y_0\|_{L^2(\Omega)}+\|\sigma\|_{L^{\infty}(0,T;L^2(\Omega))}\|\tau\|_{\mathfrak{B}[0,T]}+k^{\frac{1}{2}}\|u\|_{L^2(0,T;L^2(\Omega))}\Big).
\end{equation*}
To estimate $F_2$, we first use the triangle inequality  and Lemma \ref{rh:projection:operator} to have
\begin{align}
\|\eta^{n-1}-P_k^n\mathcal{P}_h^1\eta\|_{L^2(\Omega)}&\leq \|\eta^{n-1}-P_k^n\eta\|_{L^2(\Omega)}+\|P_k^n\eta-P_k^n\mathcal{P}_h^1\eta\|_{L^2(\Omega)}\nonumber\\&\leq\|\eta^{n-1}-P_k^n{\eta}\|_{L^2(\Omega)}+Ch^{2s}\| P_k^n{\eta}\|_{H^{1+s}(\Omega)}.\label{fully:main:eq:2}
\end{align}
We know that
\begin{equation}
\|\eta^{n-1}-P_k^n\eta\|_{L^2(\Omega)}\leq Ck_n^{\frac{1}{2}}\|\eta_t\|_{L^2(I_n;L^2(\Omega))},\label{fully:main:eq:3}
\end{equation}
and
\begin{equation}
\| P_k^n\eta\|_{H^{1+s}(\Omega)}\leq C k_n^{-\frac{1}{2}}\|\eta\|_{L^2(I_n;H^{1+s}(\Omega))}.\label{fully:main:eq:4}
\end{equation}
Using  (\ref{fully:main:eq:3})-(\ref{fully:main:eq:4}) in (\ref{fully:main:eq:2}), we get
\begin{align*}
|F_2|&\leq C\Big(\sum_{n=1}^{N}h^{4s}\|  P_k^n\eta\|_{H^{1+s}(\Omega)}^2+k_n\|\eta_t\|^2_{L^2(I_n;L^2(\Omega))}\Big)^{\frac{1}{2}}\nonumber\\&\leq C\Big(\sum_{n=1}^{N}h^{4s}k_{n}^{-1}\| P_k^n{\eta}\|_{L^2(I_n;H^{1+s}(\Omega))}^2+k_{n}\|\eta_t\|^2_{L^2(I_n;L^2(\Omega))}\Big)^{\frac{1}{2}}\nonumber\\&\leq C(h^{2s}k^{-\frac{1}{2}}+k^{\frac{1}{2}})\|g\|_{L^2(0,T;L^2(\Omega))},
\end{align*}
the last inequality is obtained by use of  Proposition \ref{psi:stability}.
Combine the bounds of $F_1$ and $F_2$ we find that
\begin{align}
|I_1|&\leq C(h^{2s}k^{-\frac{1}{2}}+k^{\frac{1}{2}}) \Big(k^{\frac{1}{2}}h^{-1}\|y_0\|_{L^2(\Omega)}+\|\sigma\|_{L^{\infty}(0,T;L^2(\Omega))}\|\tau\|_{\mathfrak{B}[0,T]}\nonumber
\\&\quad\quad+k^{\frac{1}{2}}\|u\|_{L^2(0,T;L^2(\Omega))}\Big)\|g\|_{L^2(0,T;L^2(\Omega))}.\label{required:for:estimate:1}
\end{align}
For $I_2$, we observe that 
\begin{align}
|I_{2}|&=\Big|\langle\sigma\tau,\eta\rangle_{\Omega_T}-\sum_{n=1}^N\int_{t_{n-1}}^{t_n} \langle \sigma\tau,P_k^n\mathcal{P}_h^1\eta\rangle_{I_n}\,dt\Big|\notag\\
&=\Big|\sum_{n=1}^{N}\int_{t_{n-1}}^{t_n}\int_{\Omega}\sigma(x,t)(\eta-P_k^n\mathcal{P}_h^1\eta)(x)\,dxd\tau(t)\Big|\nonumber\\
&\leq  C\|\sigma\|_{L^{\infty}(0,T;L^2(\Omega))}\|\tau\|_{\mathfrak{B}[0,T]}\|\eta-P_k^n\mathcal{P}_h^1\eta\|_{L^{\infty}(0,T;L^2(\Omega))}.\label{fully:main:eq:6}
\end{align}

\noindent
Since 
\begin{align}
\|\eta-P_k^n\mathcal{P}_h^1\eta\|_{L^{\infty}(0,T;L^2(\Omega))}&\leq \|\eta-P_k^n{\eta}\|_{L^{\infty}(0,T;L^2(\Omega))}+\|P_k^n\eta-P_k^n\mathcal{P}_h^1\eta\|_{L^{\infty}(0,T;L^2(\Omega))}\nonumber\\
&\leq Ck^{\frac{1}{2}}\|\eta_t\|_{L^2(0,T;L^2(\Omega))}+Ch^{s}\|P_k^n\eta\|_{L^{\infty}(0,T;H^1(\Omega))}\nonumber\\
&\leq C(k^\frac{1}{2}+h^{s})\|g\|_{L^2(0,T;L^2(\Omega))},\label{fully:main:eq:7}
\end{align}
substitution of  (\ref{fully:main:eq:7}) in (\ref{fully:main:eq:6}) implies
\begin{align}
|I_2|
\leq  C(k^{\frac{1}{2}}+h^{s})\|\sigma\|_{L^{\infty}(0,T;L^2(\Omega))}\|\tau\|_{\mathfrak{B}[0,T]}\|g\|_{L^2(0,T;L^2(\Omega))}.\label{required:for:main:estimate:2}
\end{align}
For $I_3$, use of  duality pairing and Lemma \ref{projection:lh} to have
\begin{align}
|I_3|&\leq
 Ch^{s}\|y_0\|_{L^2(\Omega)}\|g\|_{L^2(0,T;L^2(\Omega))}.\label{required:for:main:estimate:3}
\end{align}
Finally, apply the Cauchy-Schwarz inequality to bound  $I_4$ as
\begin{align}
|I_4|
&\leq \|u\|_{L^2(0,T;L^2(\Omega))}\|\eta-P_k^n\mathcal{P}_h^1\eta\|_{L^2(0,T;L^2(\Omega))}\nonumber\\
&\leq C(h^{s}+k^{\frac{1}{2}})\|u\|_{L^2(0,T;L^2(\Omega))}\|g\|_{L^2(0,T;L^2(\Omega))},\label{required:for:main:estimate:5}
\end{align}
where we have used  $\|\eta- P_k^n \mathcal{P}_h^1\eta\|_{L^2(0,T;L^2(\Omega))}\leq C\|\eta-P_k^n\mathcal{P}_h^1\eta\|_{L^\infty(0,T;L^2(\Omega))}$ and (\ref{fully:main:eq:7}). 
Combine (\ref{required:for:estimate:1}), (\ref{required:for:main:estimate:2})-(\ref{required:for:main:estimate:5}) together with (\ref{rearranging:fully:discrete:1}), and the definition of  $L^2(0,T;L^2(\Omega))$-norm produces the desired estimate (\ref{fuu:es1}).

\noindent
 To prove the estimate  (\ref{fuu:es2}), first we introduce the following auxiliary problem: For $\tilde{g}\in L^2(0,T;L^2(\Omega))$, find $\xi\in \hat{\mathcal{X}}(0,T)$ such that 
 \begin{align}
 \begin{cases}
 \frac{\partial \xi}{\partial t}-\Delta \xi=\tilde{g}\;\;\text{in}\;\Omega_T,\\
 \xi=0\;\;\;\text{on}\;\Gamma_T,\\
 \xi(\cdot,0)=0\;\;\text{in}\;\;\;\Omega.
 \end{cases}\label{rev:1}
 \end{align} 
An application  of Proposition \ref{psi:stability} gives
\begin{align}
\|\xi\|_{H^1(0,T;L^2(\Omega))}+\|\xi\|_{L^2(0,T; H^{1+s}(\Omega))}\leq C_R\|\tilde{g}\|_{L^2(0,T;L^2(\Omega))}.\label{xi:stability}
\end{align}
 Set $\tilde{g}=\phi-\Phi_h(u)$ in (\ref{rev:1}). Then multiply  the resulting equation by $\phi$  and use integration by parts  formula to have
 \begin{align}
 \int_{\Omega_T}(\phi&-\Phi_h(u))\tilde{g}\,dxdt=\int_{0}^T\int_{\Omega}(\phi-\Phi_h(u))(\frac{\partial \xi}{\partial t}-\Delta \xi)\,dxdt\notag\\
 &=(\phi,\frac{\partial \xi}{\partial t})_{\Omega_T}+(\nabla \phi,\nabla \xi)_{\Omega_T}-\sum_{n=1}^N\int_{t_{n-1}}^{t_n}\Big\{k_n^{-1}(\phi_h^{n-1}(u),\xi^n-\xi^{n-1})+(\nabla \phi_h^{n-1}(u),\nabla \xi)\Big\}\,dt\notag\\
  &=(y-y_d,\xi)_{\Omega_T}+\sum_{n=1}^N\int_{t_{n-1}}^{t_n}\Big\{k_n^{-1}(\phi_h^{n}(u)-\phi_h^{n-1}(u),\xi^n)-(\nabla \phi_h^{n-1}(u),\nabla \xi)\Big\}\,dt,\label{rev:c:1}
 \end{align}
where  in the last step we have  utilized (\ref{adjoint:eq:in:time}). 
 Notice that, from  (\ref{inter:fully:discrete:costate})  we get 
 \begin{align}
 -\sum_{n=1}^N(D_{k_n}\phi_h^n(u),P_k^n\mathcal{P}_h^1\xi)+(\nabla \phi_h^{n-1}(u),\nabla P_k^n\mathcal{P}_h^1\xi)=\sum_{n=1}^N(y_h^n(u)-P_k^ny_d,P_k^n\mathcal{P}_h^1\xi).\label{rev:c:2}
 \end{align}
 Utilize  (\ref{rev:c:1}) and (\ref{rev:c:2}) to obtain
 \begin{align}
  \int_{\Omega_T}(\phi&-\Phi_h(u))\tilde{g}\,dxdt=(y-y_d,\xi)_{\Omega_T}-\sum_{n=1}^N\int_{t_{n-1}}^{t_n}(y_h^n(u)-P_k^ny_d,P_k^n\mathcal{P}_h^1\xi)\,dt\notag\\\quad&+\sum_{n=1}^N\int_{t_{n-1}}^{t_n}\Big\{k_n^{-1}(\phi_h^n(u)-\phi_h^{n-1}(u),\xi^n-P_k^n\mathcal{P}_h^1\xi)-(\nabla \phi_h^{n-1}(u),\nabla(\xi-P_k^n\mathcal{P}_h^1(u)))\Big\}\,dt\notag\\&=:\tilde{I}_1+\tilde{I}_2.\label{rev:c:3}
 \end{align}
 The estimate of  $\tilde{I}_2$ follows by argument similar to the  proof of $I_1$ and hence use of (\ref{xi:stability}) leads to
 \begin{align}
| \tilde{I}_2|\leq C(h^{2s}k^{-\frac{1}{2}}+k^{\frac{1}{2}})\|\tilde{g}\|_{L^2(0,T;L^2(\Omega)}\Big(\|y_h^n(u)\|_{L^2(0,T;L^2(\Omega))}+\|P_k^n y_d\|_{L^2(0,T;L^2(\Omega))}\Big),\label{rev:I1}
 \end{align}
 where we have utilized the stability estimate of (\ref{inter:fully:discrete:costate}), which is easily obtained by an application of Lemma \ref{fully:discrete:inter:lemma:1}, is stated as
 \begin{align*}
 \sum_{n=1}^N\|\phi_h^n(u)-\phi_h^{n-1}(u)\|_{L^2(\Omega)}^2+Ck\|\phi_h^0(u)\|_{H^1(\Omega)}^2\leq Ck_n\Big(\|y_h^n(u)\|_{L^2(0,T;L^2(\Omega))}+\|P_k^ny_d\|_{L^2(0,T;L^2(\Omega))}\Big).
 \end{align*}
 An application of the Cauchy-Schwarz inequality together with  a simple calculation  to have
 \begin{align}
 |\tilde{I}_1|&=|(y-y_d,\xi)_{\Omega_T}-\sum_{n=1}^N\int_{t_{n-1}}^{t_n}(y_h^n(u)-P_k^n y_d,P_k^n\mathcal{P}_h^1\xi)\,dt|\notag\\&\leq \Big(\|y-y_h^n(u)\|_{L^2(0,T;L^2(\Omega))}+\|y_d-P_k^n y_d\|_{L^2(0,T;L^2(\Omega))}\Big)
 \|\xi\|_{L^2(0,T;L^2(\Omega))}\notag\\&\quad+\Big(\|y_h^n(u)\|_{L^2(0,T;L^2(\Omega))}+\|P_k^ny_d\|_{L^2(0,T;L^2(\Omega))}\Big)\|\xi-P_k^n\mathcal{P}_h^1\xi\|_{L^2(0,T;L^2(\Omega))}.
 \end{align}
 Utilize (\ref{time:pro}),  $\|\xi-P_k^n\mathcal{P}_h^1\xi\|_{L^2(0,T;L^2(\Omega))}\leq C\|\xi-P_k^n\mathcal{P}_h^1\xi\|_{L^{\infty}(0,T;L^2(\Omega))}$  and (\ref{fully:main:eq:7}) to have
 \begin{align}
 |\tilde{I}_1|&\leq C\Big(\|y-y_h^n(u)\|_{L^2(0,T;L^2(\Omega))}+k\|y_d\|_{H^1(0,T;L^2(\Omega))}\Big)\|\tilde{g}\|_{L^2(0,T;L^2(\Omega))}\notag\\&\quad+C(h^s+k^{\frac{1}{2}})\|\tilde{g}\|_{L^2(0,T;L^2(\Omega))}\Big(\|y_h^n(u)\|_{L^2(0,T;L^2(\Omega))}+\|P_k^ny_d\|_{L^2(0,T;L^2(\Omega))}\Big)\label{rev:I2}
 \end{align}
 Combine (\ref{rev:c:3}), (\ref{rev:I1}), (\ref{rev:I2}) and use the definition of $L^2(0,T;L^2(\Omega))$-norm  to obtain (\ref{fuu:es2}). This completes the proof.
 \end{proof}
 
 \noindent
 To estimate the error in the control variable, it is required to introduce the  completely discretized control problem as
\begin{align}
\mathcal{J}_{h,k}(u_{\rho})=J_{h,k}(y_{\rho},u_{\rho}) \;\;\text{subject to}\;\; u_{\rho}\in \mathbb{U}_{ad}^n,\label{fully:control}
\end{align}
 where the discretization parameters $h$ and $k$ are gathered under the subscript $\rho$. 
The unique solution $u_{\rho}$ of (\ref{fully:control}) satisfies the optimality condition 
\begin{align}
\mathcal{J}_{h,k}'(u_{\rho})(\tilde{u}_{\rho}-u_{\rho})=\sum_{n=1}^{N}\int_{t_{n-1}}^{t_n}(\Lambda u_{\rho}+\phi_{\rho})(\tilde{u}_{\rho}-u_{\rho})\,dt\geq 0\;\;\;\;\forall \tilde{u}_{\rho}\in \mathbb{U}_{ad}^n.\label{fully:rho}
\end{align} 
Define the   $L^2$-projection $\Pi^n_d:L^2(0,T;L^2(\Omega))\rightarrow \mathbb{U}_{ad}^n$. Notice that $\Pi^n_d \mathbb{U}_{ad}\subset \mathbb{U}_{ad}^n$, $n\in[1:N]$.
 
 We  now prove the following error bound for the control variable.
\begin{theorem}\label{fully:discrete:error:in:control:time}
Let $u$ and $u_{\rho}$ be the solutions of (\ref{funct:weak}) and (\ref{fully:control}), respectively.  Consider the sufficient  optimality condition (\ref{opsc1}) is true. Then, for $s\in(\frac{1}{2},\beta)$, the estimate 
\begin{equation*}
\|u-u_{\rho}\|_{L^2(0,T;L^2(\Omega))}\leq \tilde{C}\frac{h^{s}}{\sqrt{\gamma}}+\overline{C}\frac{(h^{2s}k^{-\frac{1}{2}}+k^{\frac{1}{2}}+h^{s})}{\gamma} 
\end{equation*}
is valid. Here, $\tilde{C}$   and $\bar{C}$ are 
given by
\begin{align} 
\tilde{C}&:=C\Big(\|\sigma\|_{L^{\infty}(0,T;L^2(\Omega))},\; \|\tau\|_{\mathfrak{B}[0,T]},\;\|y_0\|_{L^2(\Omega)},\;\|y_d\|_{L^2(0,T;L^2(\Omega))},\Lambda\Big).\label{tildec}\\
\overline{C}&=C\Big(\|y_0\|_{L^2(\Omega)},\;\|\sigma\|_{L^\infty(0,T;L^2(\Omega))},\;\|\tau\|_{\mathfrak{B}[0,T]},\|u\|_{L^2(0,T;L^2(\Omega))},\;\|y_d\|_{H^1(0,T;L^2(\Omega))},T\Big).\label{cbar}
\end{align}

\end{theorem}
\begin{proof}
To prove this,  we  introduce an auxiliary problem as follows: Find $\tilde{u}_h^n\in \mathbb{U}_{ad}^n$ such that
\begin{equation}
\min_{u_{h}^{n}\in \mathbb{U}_{ad}^n} \mathcal{J}(u_{h}^n),\label{control:eq3}
\end{equation}
here, only the control variable is discretized. 
As a result, the optimality condition 
\begin{align}
\mathcal{J}'(\tilde{u}_h^n)(\Pi_d^n u-\tilde{u}_h^n)\geq 0\;\;\;\;\forall\Pi_d^n u\in\mathbb{U}_{ad}^n,\label{inter:control}
\end{align}
is satisfied.
We decompose the error as follows:
\begin{equation}
u-u_{\rho}=(u-\tilde{u}_{h}^n)+(\tilde{u}_{h}^n-u_{\rho}). \label{control:eq4}
\end{equation}
To estimate the first term of (\ref{control:eq4}), utilize (\ref{opsc1})  for any $\tilde{u}\in U_{ad}$ to have
\begin{align*}
\gamma\|u-\tilde{u}_{h}^n\|^{2}_{L^2(0,T;L^2(\Omega))}&\leq \mathcal{J}''(\tilde{u})(u-\tilde{u}_{h}^n,u-\tilde{u}_{h}^n)\\
&=\mathcal{J}'(u)(u-\tilde{u}_{h}^n)-\mathcal{J}'(\tilde{u}_{h}^n)(u-\tilde{u}_{h}^n)\\&=\mathcal{J}'(u)(u-\tilde{u}_{h}^n)-
\mathcal{J}'(\tilde{u}_{h}^n)(u-\Pi_d^n u)-\mathcal{J}'(\tilde{u}_{h}^n)(\Pi_d^n u-\tilde{u}_{h}^n).
\end{align*}
An application of (\ref{optimality:condition:exp})  and (\ref{inter:control}) yields 
\begin{equation*}
\mathcal{J}'(u)(u-\tilde{u}_{h}^n)\leq 0\;\;\;\text{and}\;\;\;-\mathcal{J}'(\tilde{u}_{h}^n)(\Pi_d^n u-\tilde{u}_{h}^n)\leq 0.
\end{equation*}
 The properties of $\Pi_d^n$ and the  Young's inequality result in
\begin{align*}
\gamma\|u-\tilde{u}_{h}^n\|^{2}_{L^2(0,T;L^2(\Omega))}&\leq-\mathcal{J}'(\tilde{u}_{h}^n)(u-\Pi_d^n{u})\\&=-\int_{0}^{T}\Big(\Lambda\tilde{u}_{h}^n+\phi(\tilde{u}_{h}^n),u-\Pi_d^n u\Big)\,dt
\\&=-\int_{0}^{T}(\phi(\tilde{u}_{h}^n)-\Pi_d^n \phi(\tilde{u}_{h}^n),u-\Pi_d^n u)\,dt
\\&\leq\int_{0}^{T}\Big\{\frac{1}{2}\|\phi(\tilde{u}_{h}^{n})
-\Pi_d^n \phi(\tilde{u}_{h}^n)\|^{2}_{L^2(\Omega)}+\frac{1}{2}\|u-\Pi_d^n u\|^{2}_{L^2(\Omega)}\Big\}\,dt.
\end{align*}
Use of  Lemma \ref{projection:lh} gives
\begin{align*}
\|u-\tilde{u}_{h}^n\|_{L^2(0,T;L^2(\Omega))}&\leq \int_{0}^{T}\Big\{\frac{C}{\sqrt{\gamma}}h^{s}\| \phi(\tilde{u}_{h}^n)\|_{H^1(\Omega)}+\frac{C}{\sqrt{\gamma}}h^{s}\|{u}\|_{H^1(\Omega)}\Big\}\,dt
\leq  \frac{\tilde{C}}{\sqrt{\gamma}}h^{s}
\end{align*}
with $\tilde{C}$ is defined in (\ref{tildec}). 
Using the optimality condition (\ref{fully:rho}),
we notice that
\begin{align*}
\mathcal{J}_{h,k}'(u_{\rho})(u_{\rho}-\tilde{u}_h^n) \; \leq \;0 \;\leq \;\mathcal{J}'(\tilde{u}_{h}^n)(u_{\rho}-\tilde{u}_{h}^n).
\end{align*}
We have the second-order optimality condition for the problem (\ref{fully:control}) as
\begin{align}
\mathcal{J}_{h,k}''(u_{\rho})(\hat{u}_h^n,\hat{u}_h^n)\geq \gamma\|\hat{u}_h^n\|^{2}_{L^2(0,T;L^2(\Omega))}.\label{control:eq2}
\end{align}
Utilize (\ref{control:eq2}) for any  $\hat{u}_{h}^n\in \mathbb{U}_{ad}^n$ to obtain
\begin{align*}
\gamma\|u_{\rho}-\tilde{u}_{h}^n\|^{2}_{L^2(0,T;L^2(\Omega))}&\leq \mathcal{J}_{h,k}''(\hat{u}_h^n)(u_{\rho}-\tilde{u}_{h}^n,{u}_{\rho}-\tilde{u}_{h}^n)\\&= \mathcal{J}_{h,k}'(u_{\rho})(u_{\rho}-\tilde{u}_{h}^n)- \mathcal{J}_{h,k}'(\tilde{u}_{h}^n)(u_{\rho}-\tilde{u}_{h}^n)\\&\leq \mathcal{J}'(\tilde{u}_h^n)(u_{\rho}-\tilde{u}_{h}^n)-\mathcal{J}_{h,k}'(\tilde{u}_{h}^n)(u_{\rho}-\tilde{u}_{h}^n)
\\&\leq \overline{C}(h^{2s}k^{-\frac{1}{2}}+k^{\frac{1}{2}}+h^{s})\|{u}_{\rho}-\tilde{u}_{h}^n\|_{L^2(0,T;L^2(\Omega))},
\end{align*}
where the last step follows by  using (\ref{optimality:condition:exp}), (\ref{fully:rho}) and  (\ref{fuu:es2}), which completes the rest of the proof.
\end{proof}
Now, we are in a position to estimate the error in the state variable in the $L^2(0,T;L^2(\Omega))$-norm. 
\begin{theorem}\label{theorem:for:state}
Under the assumption of Theorem \ref{stability}, let
$y$ and ${Y}_h$ be the solutions  (\ref{state:control:eq:1}), and  (\ref{fully:discrete:state}), respectively. So,  for $s\in(\frac{1}{2},\beta)$ the following is true:
\begin{align*}
\|y-{Y}_h\|_{L^2(0,T;L^2(\Omega))}&\leq C(h^{2s}k^{-\frac{1}{2}}+k^{\frac{1}{2}}+h^{s})\Big\{\|y_0\|_{L^2(\Omega)}+\|\sigma\|_{L^{\infty}(0,T;L^2(\Omega))}\|\tau\|_{\mathfrak{B}[0,T]}+\|u\|_{L^2(0,T;L^2(\Omega))}\Big\}\nonumber\\&\quad\quad+C\|u-u_h^n\|_{L^2(0,T;L^2(\Omega))}.
\end{align*}
\end{theorem}
\begin{proof}
Let $\eta$ solves the problem (\ref{adjoint}) with $g\in L^2(0,T;L^2(\Omega))$. The duality argument with (\ref{state:control:eq:1}) leads to the assertion that 
\begin{align}
\int_{\Omega_T}(y-{Y}_h)g\,dxdt&=\int_0^T\int_{\Omega}(y-{Y}_h)(-\frac{\partial \eta}{\partial t}-\Delta\eta)\,dxdt\nonumber
\\&=-(y,\frac{\partial \eta}{\partial t})_{\Omega_T}+(\nabla y,\nabla\eta)_{\Omega_T}+\sum_{n=1}^{N}\int_{t_{n-1}}^{t_n}\{(y_h^n,\frac{\partial \eta}{\partial t})-(\nabla y_h^n,\nabla \eta)\}\,dt\nonumber
\\
&=-\sum_{n=1}^{N}\int_{t_{n-1}}^{t_n}\{{k_n}^{-1}(y_h^n-y_h^{n-1},\eta^{n-1})+(\nabla y_h^n,\nabla \eta)\}\,dt+(y_0-y_{h,0},\eta(\cdot,0))\nonumber\\
&\quad\quad+\langle\sigma\tau,\eta\rangle_{\Omega_T}+(u,\eta)_{\Omega_T}.\label{F:1}
\end{align}
From (\ref{fully:discrete:state}), we have 
\begin{equation}
\sum_{n=1}^{N}\{(D_{k_n}y_h^n,P_k^n\mathcal{P}_h^1\eta)+(\nabla y_h^n,\nabla P_k^n\mathcal{P}_h^1\eta)\}=\sum_{n=1}^{N}\langle\sigma\tau,P_k^n\mathcal{P}_h^1\eta\rangle_{I_n}+\sum_{n=1}^{N}(u_h^n,P_k^n\mathcal{P}_h^1\eta).\label{main:state:error:eq:1:2}
\end{equation}
 Utilize (\ref{F:1}) and  (\ref{main:state:error:eq:1:2}) to  achieve  that
\begin{align}
\int_{\Omega_T}(y-{Y}_h)g\,dxdt&=-\sum_{n=1}^{N}\int_{t_{n-1}}^{t_n}\Big\{k_{n}^{-1}(y_h^n-y_h^{n-1},\eta^{n-1}-P_k^n\mathcal{P}_h^1\eta)+(\nabla y_h^n,\nabla(P_k^n\eta-P_k^n\mathcal{P}_h^1\eta))\Big\}\,dt\nonumber\\&\quad\quad+\Big\{\langle\sigma\tau,\eta\rangle_{\Omega_T}-\sum_{n=1}^{N}\int_{t_{n-1}}^{t_n}\langle\sigma\tau,P_k^n\mathcal{P}_{h}^1\eta\rangle_{I_n}dt\Big\}+\{(y_0-y_{h,0},\eta(\cdot,0))\}\nonumber\\&\quad\quad+\Big\{(u,\eta)_{\Omega_T}-\sum_{n=1}^{N}\int_{t_{n-1}}^{t_n}(u_h^n,P_k^n\mathcal{P}_h^1\eta)\,dt\Big\}\nonumber\\&=:I_1+I_2+I_3+\tilde{I}_4.\label{rearranging:fully:discrete}
\end{align}
The bounds of $I_1,\;I_2,\;I_3$ are found in Lemma \ref{inter:error:estimate:theorem:for:state:time}. Now, we estimate  $\tilde{I}_4$ as follows
 \begin{align}
|\tilde{I}_4|
&=\Big|\sum_{n=1}^{N}\int_{t_{n-1}}^{t_n}(u,\eta)dt-\sum_{n=1}^{N}\int_{t_{n-1}}^{t_n}(u_h^n,P_k^n\mathcal{P}_h^1\eta)\,dt\Big|\nonumber\\
&=\Big|\sum_{n=1}^{N}\int_{t_{n-1}}^{t_n}(u,\eta-P_k^n\mathcal{P}_h^1\eta)dt+\sum_{n=1}^{N}\int_{t_{n-1}}^{t_n}(u-u_h^n,P_k^n\mathcal{P}_h^1\eta)\,dt\Big|.\nonumber
\end{align}
The Cauchy-Schwarz inequality, Poincar\'{e} inequality, Proposition \ref{psi:stability} and  (\ref{fully:main:eq:7}) together with $\|\nabla P_k^n{\mathcal{P}}_h^1\eta\|_{L^2(0,T;L^2(\Omega))} \leq C\|\nabla\eta\|_{L^2(0,T;L^2(\Omega))}$   implies
\begin{align*}
|\tilde{I}_4|&\leq  C\Big(\|u\|_{L^2(0,T;L^2(\Omega))}\|\eta-P_k^n\mathcal{P}_h^1\eta\|_{L^\infty(0,T;L^2(\Omega))}+\|u-u_h^n\|_{L^2(0,T;L^2(\Omega))}\|\nabla\eta\|_{L^2(0,T;L^2(\Omega))}\Big)\nonumber\\
&\leq  C\Big\{(h^{s}+k^{\frac{1}{2}})\|u\|_{L^2(0,T;L^2(\Omega))}+\|u-u_h^n\|_{L^2(0,T;L^2(\Omega))}\Big\}\|g\|_{L^2(0,T;L^2(\Omega))},
\end{align*}
where we have used $\|\eta-P_k^n\mathcal{P}_h^1\eta\|_{L^2(0,T;L^2(\Omega))}\leq C\|\eta-P_k^n\mathcal{P}_h^1\eta\|_{L^\infty(0,T;L^2(\Omega))}$.
Combining all the estimates of $I_1$, $I_2$, $I_3$ and $\tilde{I}_4$ together with (\ref{rearranging:fully:discrete}) and the definition of $L^2(0,T;L^2(\Omega))$-norm yields the desired estimate.
\end{proof}
The next theorem is devoted to the error estimate for the co-state variable. We omit the details of the proof. 
\begin{theorem}\label{final:thm:co}
Let $\phi$ and $\Phi_h$ be the solutions of (\ref{adjoint:eq:in:time}) and (\ref{fully:discrete:costate})-(\ref{fully:final:time}), respectively. Then, we have
\begin{align*}
\|\phi-\Phi_h\|_{L^2(0,T;L^2(\Omega))}&\leq C(h^{2s}k^{-\frac{1}{2}}+k^{\frac{1}{2}}+h^s)\Big\{\|y_h^n\|_{L^2(0,T;L^2(\Omega))}+\|P_k^ny_d\|_{L^2(0,T;L^2(\Omega))}\Big\}\nonumber\\&\quad\quad+k\|y_{d,t}\|_{L^2(0,T;L^2(\Omega))}+\|y-{Y}_h\|_{L^2(0,T;L^2(\Omega))}.
\end{align*}
\end{theorem}

\begin{remark} With $k=\mathcal{O}(h^{2s})$, Theorems \ref{fully:discrete:error:in:control:time},  \ref{theorem:for:state} and \ref{final:thm:co} yield the following error estimate

\begin{align*}
\|y-{Y}_h\|_{L^2(0,T;L^2(\Omega))}+\|u-u_{{\rho}}\|_{L^2(0,T;L^2(\Omega))}+\|\phi-\Phi_h\|_{L^2(0,T;L^2(\Omega))}\leq Ch^{s},
\end{align*}
where $\rho$   acquires  the discretization parameters $h,\,k$.
\end{remark}
\section{A posteriori error estimates}\label{apos}
The a posteriori error bounds for the  control, co-state, and state variables are derived in this section.

 For $n\in[1:N]$,   $Y_h^0(x)=y_{h,0}(x)$ and $\Phi^N_h(x)=0$,  we recast the optimality conditions (\ref{fully:optimal:state})- (\ref{fully:discrete:optimality:condition}) in terms of $Y_h$, $\Phi_h$, $U_h$ as
\begin{align}
(D_{k_n}Y_h^n,w_h)+(\nabla Y_h^n,\nabla w_h)&=\langle\sigma\tau,w_h\rangle_{I_n}+(U_h,w_h)\;\;\forall w_h\in \mathbb{W}_h^n,\label{ful:state:2:1}\\
-(D_{k_n} \Phi_h^n,w_h)+(\nabla \Phi_h^{n-1},w_h)&=(Y_h^n-P_k^ny_{d},w_h)\;\;\forall w_h\in \mathbb{W}_h^n,\label{ful:costate:2:1}\\
(\Lambda U_h+\Phi_h^{n-1},\hat{u}_h^n-U_h)&\geq 0\;\;\;\;\forall \hat{u}_h^n\in \mathbb{U}_{ad}^{n}.\label{ful:optimality:condition:2:1}
\end{align}
In the following, we recall two  lemmas that are essential for obtaining a posteriori error bounds. 
 \begin{lemma}\label{lagrange:K}
\cite{ciar78}. For $w\in H^{1+s}(K)$, let $K\in\mathcal{T}_h$ 
and $m=0\; \text{or}\;1$.
 Then
\begin{align*}
\|w-\Pi_h w\|_{H^{m}(K)}\leq C_{I,m}h_{K}^{1+s-m}|w|_{H^{1+s}(K)}\;\;\forall w\in H^{1+s}(K),
\end{align*}
where  $\Pi_h:\mathcal{C}_0(\overline{\Omega})\rightarrow \mathbb{W}_h^0$ be the nodal interpolation operator.
\end{lemma}
\begin{lemma}\label{lagrange:l}
\cite{kufner77} For $K\in\mathcal{T}_h$, $1\leq p<\infty$, we have
\begin{align*}
\|w\|_{L^{p}(e)}\leq C_{I,e}(h_K^{-1/p}\|w\|_{L^{p}(K)}+h_K^{1-1/p}|w|_{W^{1,p}(K)})\;\;  \forall w\in W^{1,p}(\Omega).
\end{align*}
\end{lemma}
\noindent
The error between $y$ and $Y_h$ is determined 
with the help of intermediate error estimates. For this, we introduce the auxiliary problems:
For $U_h\in \mathbb{U}_{ad}^n$, let  $y(U_h)\in W(0,T)$ be the solution of
\begin{align}
-(y(U_h), \frac{\partial w}{\partial t})_{\Omega_T}+(\nabla y(U_h),\nabla w)_{\Omega_T}&=\langle \sigma\tau,w\rangle_{\Omega_T}+(y_0,w(\cdot,0))+(U_h,w)_{\Omega_T}\;\;\;\;\forall w\in \mathcal{X}(0,T)\label{inter:state:time}
\end{align}
with $w(\cdot,T)=0$, and let $\phi(U_h)\in \hat{\mathcal{X}}(0,T)$ satisfy
\begin{align}
\begin{cases}
-(\frac{\partial \Phi(U_h)}{\partial t}, w)_{\Omega_T}+(\nabla \phi(U_h),\nabla w)_{\Omega_T}=(y(U_h)-y_{d},w)_{\Omega_T}\;\;\;\;\forall w\in L^2(0,T;H^1_0(\Omega)),\\
\phi(U_h)(\cdot,T)=0.
\end{cases}\label{inter:adjoint:time}
\end{align}
To find the error bounds,  we first split the errors and use triangle inequality to have
\begin{align}
\|Y_h^n-y\|_{L^2(0,T;L^2(\Omega))}&\leq  \|Y_h^n-y(U_h)\|_{L^2(0,T;L^2(\Omega))}+\|y(U_h)-y\|_{L^2(0,T;L^2(\Omega))},\label{Y:n}\\
\|\Phi_h-\phi\|_{L^2(0,T;L^2(\Omega))}&\leq  \|\Phi_h-\phi(U_h)\|_{L^2(0,T;L^2(\Omega))}+\|\phi(U_h)-\phi\|_{L^2(0,T;L^2(\Omega))}\label{phi:n}.
\end{align}

To begin with we first determine the error bound for the control variable. 
\begin{lemma}\label{apos:for:q:space}
Let $(y,u,\phi)$ be the solution of (\ref{funct:weak})-\ref{adjoint:eq:in:time}),  
 and let $(Y_h,U_h,\Phi_h)$ be the solution of (\ref{ful:state:2:1})-(\ref{ful:optimality:condition:2:1}). 
Assume that  $(\Lambda {U}_h+\Phi_h^{n-1})|_{K}\in H^{1}(K)$  and for $\tilde{u}\in \mathbb{U}_{ad}$, the following
 \begin{equation}
\Big|\int_{0}^{T}(\Lambda{U}_h+\Phi_{h}^{n-1},\tilde{u}-U_h)\,dt\Big|\leq C_{1}\int_{0}^{T}\sum_{K\in\mathcal{T}_h}h_K|\Lambda U_{h}+\Phi_{h}^{n-1}|_{H^1(K)}\|u-U_{h}\|_{L^2(K)}\,dt\label{tildeu:u}
\end{equation}
holds for some positive constant $C_1$.  
Then, we have
\begin{equation}
\|u-U_{h}\|_{L^2(0,T;L^2(\Omega))}^2\leq C_{2}^2\Big(\xi_1^n+\|\Phi_h-\phi(U_{h})\|_{L^2(0,T;L^2(\Omega))}^2\Big),\label{result:op}
\end{equation}
where $C_{2}=\sqrt{\frac{3}{2}}\max\{1,C_{1}\}$, $\xi_{1}^n:=\Big(\int_{0}^{T}\displaystyle\sum_{K\in\mathcal{T}_h}h_{K}^2|\Lambda U_h+\Phi_h^{n-1}|^2_{H^1(K)}\,dt\Big),$
and $\phi(U_h)$ be the solution of (\ref{inter:adjoint:time}).
\end{lemma}

\begin{proof}
 Inviting the optimality condition  (\ref{optimality:condition:exp}) we obtain
 \begin{eqnarray}
(\Lambda u,u-U_h)\leq-(\phi,u-U_h).\label{star}
\end{eqnarray}
Then use of (\ref{star}) results in
\begin{align}
\int_{0}^{T}\|u-U_h\|_{L^2(\Omega)}^2\,dt
\leq \int_0^T\{-(\frac{\phi}{\Lambda},u-U_h)-(U_h,u-U_h)\}\,dt.\nonumber
\end{align}
And hence
\begin{align}
\Lambda\int_{0}^{T}\|u-U_h\|_{L^2(\Omega)}^2\,dt&=\int_0^T\{-(\phi,u-U_h)-(\Lambda U_h,u-U_h)\}\,dt\nonumber\\
&=-\int_{0}^{T}(\Phi_h^{n-1}+\Lambda U_h,u-\tilde{u}_h^n)\,dt-\int_{0}^{T}(\Lambda U_h+\Phi_h^{n-1},\tilde{u}_h^n-U_h)\,dt\nonumber\\
&\quad\quad+\int_{0}^T(\Phi_h^{n-1}-\phi(U_h),u-U_h)\,dt+\int_0^T(\phi(U_h)-\phi,u-U_h)\,dt.\nonumber
\end{align}
An application of (\ref{ful:optimality:condition:2:1}) yields
\begin{align}
\Lambda\int_{0}^{T}\|u-U_h\|_{L^2(\Omega)}^2\,dt&\leq\int_0^T(\Lambda U_h+\Phi_h^{n-1},\tilde{u}_h^n-u)\,dt+\int_{0}^T(\Phi_h^{n-1}-\phi(U_h),u-U_h)\,dt\nonumber\\&\quad\quad+\int_0^T(\phi(U_h)-\phi,u-U_h)\,dt\nonumber\\
&=: E_1+E_2+E_3.\label{i123}
\end{align}
It is clear from the assumption 
 (\ref{tildeu:u}) that
\begin{align}
|E_{1}|&=\Big|\int_0^T(\Lambda U_h+\Phi_h^{n-1},\tilde{u}_h^n-u)\,dt\Big|\nonumber\\
&\leq \int_0^T \Big\{\sum_{K\in\mathcal{T}_h} C_{1}\,h_K|\Lambda U_h+\Phi_h^{n-1}|_{H^1(K)}\|u-U_h\|_{L^2(K)}\Big\}dt\nonumber\\
&\leq\frac{3\,C_{1}^2}{4}\xi_{1}^n+\frac{1}{4}\|u-U_h\|^2_{L^2(0,T;L^2(\Omega))}.\label{i1}
\end{align}
Additionally, it is obvious that 
\begin{align}
|E_2|&=\Big|\int_0^T(\Phi_h^{n-1}-\phi(U_h),u-U_h)\,dt\Big|\nonumber\\
&\leq \frac{3}{4}\int_0^T\|\Phi_h^{n-1}-\phi(U_h)\|^2_{L^2(\Omega)}\,dt+\frac{1}{4}\int_{0}^{T}\|u-U_h\|_{L^2(\Omega)}^2\,dt.\label{i2}
\end{align}
Utilize (\ref{state:control:eq:1}) and (\ref{inter:state:time}) to write the expression of $E_3$ as
\begin{align*}
E_3&=\int_0^T(u-U_h,\phi(U_h)-\phi)\,dt\nonumber\\
&=\int_0^T\Big\{-(y-y(U_h),\frac{\partial \phi(U_h)}{\partial t}-\frac{\partial \phi}{\partial t})+(\nabla (y-y(U_h)),\nabla (\phi(U_h)-\phi))\Big\}\,dt,\nonumber
\end{align*}
which combine with (\ref{adjoint:eq:in:time}) and (\ref{inter:adjoint:time}) gives $E_3\leq 0$. The proof is accomplished by combining the estimates of $E_1,\;E_2$ and $E_3$. 
\end{proof}

\noindent
The main theorem of this section will be derived by  use of  intermediate error estimates provided in the next two lemmas. 
\begin{lemma}\label{inter:lemma:time:apos}
Under the assumption of Theorem \ref{stability}, let $y(U_h)\in W(0,T)$  and $Y_h^n\in \mathbb{W}_h^n$ be the solutions of  (\ref{inter:state:time}) and  
 (\ref{ful:state:2:1}), respectively.  Then, for $n\in[1:N]$  we have
\begin{align*}
\sum_{n=1}^N\int_{t_{n-1}}^{t_n}\|Y_h^n-y(U_h)\|_{L^2(\Omega)}^2\,dt\leq C_{3}^2\sum_{n=1}^{N}\{k_n(\xi_2^n+\xi_3^n+\xi_5^n)+\xi_4^n\},
\end{align*}
where
\begin{align*}
\xi_2^n&:=\sum_{K\in \mathcal{T}_h^n}h_K^{2(1+s)}\|k_n^{-1}(Y_h^n-Y_h^{n-1})-\Delta Y_h^n-U_h\|^2_{L^2(K)}+\sum_{e\in\mathcal{E}_h^n}h_e^{1+2s}\Big\|\jump{\frac{\partial Y_h^n}{\partial n}}\Big\|^2_{L^2(e)},\\
\xi_3^n &:=\|Y_h^n-Y_h^{n-1}\|_{L^2(\Omega)}^2,\\
\xi_4^n&:= \|y_0-Y_h^0\|_{L^2(\Omega)}^2,\\
\xi_5^n&:=\sum_{n=1}^N\sum_{K\in \mathcal{T}_{h}^n}\Big( h_K^2\|\sigma\|_{L^{\infty}(I_n;L^2(K))}^2+k_n\|\sigma-{P}_k^n\,\sigma\|_{L^{\infty}(I_n;L^2(K))}^2 \Big)\|\tau\|_{\mathfrak{B}(I_n)}^2,
\end{align*}
and $C_{3}=C_{R}\max\{1,C_{I},C_{I,0},C_{I,2},C_{I,3},C_{I,e}C_{I,0}\}$.
\end{lemma}
\begin{proof}
 Let $\eta$ be the solution of problem (\ref{adjoint}) with $g\in L^2(0,T;L^2(\Omega))$. It should be noted that  $\eta=0$ on $\partial\Omega$, $\eta^N=\eta(\cdot,T)=0$.  Use of (\ref{inter:state:time}) and integrating by parts yields
 \begin{align}
 &\int_{\Omega_T}(Y_h^n-y(U_h))g\,dxdt=\int_0^T\int_{\Omega}(Y_h^n-y(U_h))(-\frac{\partial \eta}{\partial t}-\Delta\eta)\,dxdt\nonumber\\
 &=(y(U_h),\frac{\partial \eta}{\partial t})_{\Omega_T}-(y(U_h),-\Delta\eta)_{\Omega_T}-\sum_{n=1}^{N}\int_{t_{n-1}}^{t_n}((Y_h^n,\frac{\partial \eta}{\partial t})-(\nabla Y_h^n,\nabla \eta))\,dt\nonumber\\
 &=-\langle\sigma\tau,\eta\rangle_{\Omega_T}+(Y_{h}^{0}-y_{0},\eta(\cdot,0))-(U_h,\eta)_{\Omega_T}\nonumber\\&\quad\quad+\sum_{n=1}^{N}\int_{t_{n-1}}^{t_n}\Big\{k_n^{-1}(Y_h^n-Y_h^{n-1},\eta^{n-1})+(\nabla Y_h^n,\nabla \eta)\Big\}\,dt.\label{inter:lemma:1:eq:1}
 \end{align}
We indicate  $\Pi_h^n$ as the Lagrange interpolation operator onto $\mathbb{W}_h^n$ and define $\eta_I$ such that $\eta_I|_{I_n}:=\Pi_h^{n}({P}_k^n\eta)\in \mathbb{W}_h^n$ for each time interval $I_n$. 
From equation (\ref{ful:state:2:1}), we obtain
\begin{align}
\sum_{n=1}^{N}\Big\{k_n^{-1}(Y_h^n-Y_h^{n-1},\eta_{I})+(\nabla Y_h^n,\nabla\eta_{I})\Big\}=\sum_{n=1}^{N}\langle\sigma\tau,\eta_I\rangle_{I_n}+\sum_{n=1}^{N}(U_h,\eta_{I}).\label{inter:lemma:1:eq:2}
\end{align}
 Then we have
 \begin{align}
 \int_{\Omega_T}(Y_h^n-y(U_h))g\,dxdt&=\sum_{n=1}^{N}\int_{t_{n-1}}^{t_n}k_n^{-1}(Y_h^n-Y_h^{n-1},\eta^{n-1}-\eta_{I})\,dt+\sum_{n=1}^{N}\int_{t_{n-1}}^{t_n}\{(\nabla Y_h^n,\nabla(\eta-\eta_{I}))\nonumber\\&\quad-(U_h,\eta-\eta_{I})\}\,dt
 -(y_0-Y_h^0,\eta(\cdot,0))-\langle\sigma\tau,\eta\rangle_{\Omega_T}+\sum_{n=1}^{N}\int_{t_{n-1}}^{t_n}\langle \sigma\tau,\eta_{I}\rangle_{I_n}\,dt\nonumber\\
 &=:\tilde{\mathfrak{E}}_1+\tilde{\mathfrak{E}}_2+\tilde{\mathfrak{E}}_3+\tilde{\mathfrak{E}}_4\label{inter:lemma:1:eq:3}
  \end{align}
 The terms $\tilde{\mathfrak{E}}_i|_{i=1,\ldots,4}$ are estimated separately.  
Using Lemma \ref{lagrange:K} and integration by parts we arrive at 
\begin{align}
|\tilde{\mathfrak{E}}_1|&=\Big|\sum_{n=1}^{N}\int_{t_{n-1}}^{t_n}k_n^{-1}(Y_h^n-Y_h^{n-1},\eta^{n-1}-{P}_k^n\eta+{P}_k^n(\eta-\Pi_h^n\eta))\,dt\Big|\nonumber\\
 &=\Big|\sum_{n=1}^{N}\int_{t_{n-1}}^{t_n}\int_{\Omega}k_n^{-1}(Y_h^n-Y_h^{n-1})(\eta^{n-1}-{P}_k^n\eta)\,dxdt\notag\\&\quad\quad+\sum_{n=1}^{N}\int_{t_{n-1}}^{t_n}\int_{\Omega}k_n^{-1}(Y_h^n-Y_h^{n-1}){P}_{k}^n(\eta-\Pi_h^n\eta)\,dxdt\Big|\nonumber\\
&\leq  \max\{C_{I},C_{I,0}\}\,\Big[k_n\Big(\|Y_h^n-Y_h^{n-1} \|^2_{L^2(\Omega)}\notag\\&\quad\quad+\sum_{K\in \mathcal{T}^h_n}h_{K}^{2(1+s)}\|k_n^{-1}(Y_h^n-Y_h^{n-1})\|^2_{L^2(K)}\Big)\Big]^{\frac{1}{2}}\|g\|_{L^2(0,T;L^2(\Omega))}, \label{inter:lemma:1:eq:4}
  \end{align}
  where we  have utilized the properties of $P_k^n$ and Proposition \ref{psi:stability}
 \begin{align*}
  &\|\eta^{n-1}-{P}_k^n\eta\|_{L^2(0,T;L^2(\Omega))}\leq C_{I}k_n\|\eta\|_{H^1(0,T;L^2(\Omega))},\\
 &\|{P}_k^n(\eta-\Pi_h^n\eta)\|_{L^2(\Omega)}\leq \|\eta-\Pi_h^n\eta\|_{L^2(\Omega)}\leq C_{I,0} \,h^{1+s}_K\|\eta\|_{L^2(0,T;H^{1+s}(\Omega))}.
 \end{align*}
 To estimate $\tilde{\mathfrak{E}}_2$, 
 the properties of $\Pi_h^n$ now yields
\begin{align}
|\tilde{\mathfrak{E}}_2|&=\Big|\sum_{n=1}^{N}\int_{t_{n-1}}^{t_n}\{(\nabla Y_h^n,\nabla(\eta-\eta_I))-(U_h,\eta-\eta_{I})\}\,dt\Big|\nonumber\\
&\leq\sum_{n=1}^{N}\int_{t_{n-1}}^{t_n}\Big(\sum_{K\in\mathcal{T}_h^n}\|-\Delta Y_h^n-U_h\|_{L^2(K)}\|\eta-\Pi_h^n\eta\|_{L^2(K)}\nonumber\\&\quad\quad+\sum_{e\in\mathcal{E}_h^n}\Big\|\jump{\frac{\partial Y_h^n}{\partial n}}\Big\|_{L^2(e)}\|\eta-\Pi_h^n\eta\|_{L^2(e)}\Big)\,dt\nonumber\\
&\leq \max\{C_{I,0},C_{I,e}C_{I,0}\}\Big[\sum_{n=1}^{N}k_n\Big(\sum_{K\in\mathcal{T}_h^n}h_{K}^{2(1+s)}\|-\Delta Y_h^n-U_h\|^2_{L^2(K)}\nonumber\\&\quad\quad+\sum_{e\in\mathcal{E}_h^n}h_e^{1+2s}\Big\|\jump{\frac{\partial Y_h^n}{\partial n}}\Big\|^2_{L^2(e)}\Big)\Big]^{\frac{1}{2}}\|\eta\|_{L^2(0,T;H^{1+s}(\Omega))}.\label{inter:lemma:1:eq:5}
\end{align}

\noindent 
The bound of $\tilde{\mathfrak{E}}_3$ is obtained by the application of the Cauchy-Schwarz inequality 
\begin{equation*}
|\tilde{\mathfrak{E}}_3|\leq \|Y_h^0-y_0\|_{L^2(\Omega)}\|\eta(\cdot,0)\|_{L^2(\Omega)}.\label{inter:lemma:1:eq:6}
\end{equation*}
 For $\tilde{\mathfrak{E}}_4$, adding and subtracting the suitable terms together with the definition of $P_k^n$ we obtain
\begin{align*}
|\tilde{\mathfrak{E}}_4|&=\Big|\langle\sigma\tau,\eta\rangle_{\Omega_T}-\sum_{n=1}^N\int_{t_{n-1}}^{t_{n}}\langle\sigma \tau,\eta_{I}\rangle_{I_n}\,dt\Big|
\\&=\Big|\sum_{n=1}^N\int_{\Omega}\int_{t_{n-1}}^{t_n}\sigma(x,t)\, \eta(x,t) d\tau(t)\,dx-\sum_{n=1}^{N}\int_{\Omega}\int_{t_{n-1}}^{t_n}\sigma(x,t)\Pi_h^n(P_k^n\eta)d\,\tau(t)\,dx\Big|.
\end{align*}
Hence
\begin{align}
|\tilde{\mathfrak{E}}_4|&=\Big|\sum_{n=1}^N\bigg\{\int_{\Omega}\int_{t_{n-1}}^{t_n}\sigma(x,t)(\eta-P_k^n\eta)\,d\tau(t)dx+\int_{\Omega}\int_{t_{n-1}}^{t_n}\sigma(x,t)P_k^n(\eta-\Pi_h^n\eta)\,d\tau(t)dx\bigg\}\Big|\nonumber
\\
&\leq  \sum_{n=1}^N \Big\{\sum_{K\in\mathcal{T}_h^n}\Big(\|\sigma-P_{k}^n\,\sigma\|_{L^{\infty}(I_n;L^2(K))}\|\eta-{P}_{k}^n\eta\|_{L^{\infty}(I_n;L^2(K))}\|\tau\|_{\mathfrak{B}(I_n)}\nonumber\\&\quad\quad+\|\sigma\|_{L^{\infty}(I_n;L^2(K))}\|P_k^n(\eta-\Pi_h^n\eta)\|_{L^{\infty}(I_n;L^2(K))}\|\tau\|_{\mathfrak{B}(I_n)}\Big)\Big\}\notag\\
&\leq   \max\{C_{I,2},C_{I,3}\} \sum_{n=1}^N\Big\{\sum_{K\in \mathcal{T}_{h}^n}\Big(k_n\|\sigma-P_k^n\sigma\|_{L^{\infty}(I_n;L^2(K))}^2\|\tau\|_{\mathfrak{B}(I_n)}^2\nonumber\\&\quad\quad+h_K^2\|\sigma\|_{L^{\infty}(I_n;L^2(K))}^2 \|\tau\|_{\mathfrak{B}(I_n)}^2 \Big)\Big\}^{\frac{1}{2}}
\|g\|_{L^2(0,T;L^2(\Omega))},\label{inter:fully:discrete:time:2}
\end{align}
where we used the Proposition \ref{psi:stability} and the following  properties
\begin{align*}
\|\eta-P_k^n\eta\|_{L^{\infty}(I_n;L^2(K))}&\leq C_{I,2}\,k^{\frac{1}{2}}_n\|\eta\|_{H^1(0,T;L^2(K))},\\\;\;\;\;\;\|\eta-\Pi_h^n\eta\|_{L^{\infty}(I_n;L^2(K))}&\leq C_{I,3}\,h_{K}\|\eta\|_{L^{\infty}(I_n;H^1(K))}.
\end{align*}
Combining (\ref{inter:lemma:1:eq:3}), (\ref{inter:lemma:1:eq:4}), (\ref{inter:lemma:1:eq:5}),  and (\ref{inter:fully:discrete:time:2}), we find that
\begin{align*}
&\|Y_h^n-y(U_h)\|_{L^2(0,T;L^2(\Omega))}\leq C_{3}\Bigg[\Big\{\sum_{n=1}^{N}k_n \Big(\sum_{K\in\mathcal{T}_h^n}h_{K}^{2(1+s)}\|k_n^{-1}(Y_h^n-Y_h^{n-1})-\Delta Y_h^n-U_h\|^2_{L^2(K)}
\nonumber\\&+\sum_{e\in\mathcal{E}_h^n} h_e^{1+2s}\Big\|\jump{\frac{\partial Y_h^n}{\partial n}}\Big\|_{L^2(e)}^2\Big)\Big\}^{\frac{1}{2}}+\|y_0-Y_h^0\|_{L^2(\Omega)}+\Big(\sum_{n=1}^{N}k_n\|Y^n_h-Y_h^{n-1}\|^2_{L^2(\Omega)}\Big)^{\frac{1}{2}}
\\&+\sum_{n=1}^N\bigg\{\sum_{K\in \mathcal{T}_{h}^n}\Big(k_n\|\sigma-P_k^n\,\sigma\|_{L^{\infty}(I_n;L^2(K))}^2\|\tau\|_{\mathfrak{B}(I_n)}^2+h_K^{2}\|\sigma\|_{L^{\infty}(I_n;L^2(K))}^2 \|\tau\|_{\mathfrak{B}(I_n)}^2 \Big)\bigg\}^{\frac{1}{2}}\Bigg].
\end{align*}
The proof is now complete. 
\end{proof}

\noindent
The intermediate error estimate  $\|\Phi_h-\phi(U_h)\|_{L^2(0,T;L^2(\Omega))}$ is obtained in the following lemma. 
\begin{lemma}\label{inter:lemma:phi:apos}
Let $\Phi_h$ and  $\phi(U_h)$ be the solutions of (\ref{ful:costate:2:1}) and (\ref{inter:adjoint:time}), respectively. Then, we have
\begin{align*}
\|\Phi_h-\phi(U_h)\|_{L^2(0,T;L^2(\Omega))}^2\leq C_{4}^2\sum_{i=6}^9\xi_i^n,
\end{align*}
where
\begin{align*}
\xi_6^n&:=\Big\{\int_0^T\sum_{K\in\mathcal{T}_h^n}h_{K}^{2(1+s)}\|-\Phi_{h,t}-\Delta \Phi_{h}^{n-1}-Y_h^n+P_k^ny_{d}\|^2_{L^2(K)}\,dt\notag\\&\quad\quad+
\int_0^T\sum_{e\in\mathcal{E}_h^n} h_e^{1+2s}\Big\|\jump{\frac{\partial \Phi_h^{n-1}}{\partial n}}\Big\|_{L^2(e)}dt\Big\},\\
\xi_7^n&:=\|Y_h^n-y(U_h)\|_{L^2(0,T;L^2(\Omega))}^2,\\
\xi_8^n&:=\|y_{d}-P_k^ny_{d}\|_{L^2(0,T;L^2(\Omega))}^2,\\
\xi_9^n&:=\|\Phi_h-\Phi_h^{n-1}\|_{L^2(0,T;H^1(\Omega))}^2,
\end{align*}
and $C_{4}=C_{R}\,\max\{1,\max\{C_{I,0},C_{I,0}C_{I,e}\}\}$.
\end{lemma}

\noindent
Using Lemmas \ref{apos:for:q:space}-\ref{inter:lemma:phi:apos}, we can obtain the required a posteriori error bounds.
\begin{theorem}\label{main:thm:space}
Let $(y,u,\phi)$ and $(Y_h,U_h,\Phi_h)$ be the solutions of (\ref{state:control:eq:1})-(\ref{optimality:condition:exp}) and (\ref{ful:state:2:1})-(\ref{ful:optimality:condition:2:1}), respectively. 
Assume that every requirement in Lemmas \ref{apos:for:q:space}-\ref{inter:lemma:phi:apos} is true. Then, we obtain
\begin{align*}
&\sum_{n=1}^N\int_{t_{n-1}}^{t_n}\|Y_h^n-y\|_{L^2(\Omega)}^2\,dt+\sum_{n=1}^N\int_{t_{n-1}}^{t_n}\|\Phi_h-\phi\|_{L^2(\Omega)}^2\,dt+\sum_{n=1}^{N}\int_{t_{n-1}}^{t_n}\|u-U_h\|_{L^2(\Omega)}^2\,dt\\&\leq C_{2}^2\,\xi_1^n+C_{3}^2\sum_{n=1}^N\{k_n(\xi_2^n+\xi_3^n+\xi_5^n)+\xi_4^n\}+C_{4}^2\sum_{n=1}^N\{\xi_6^n+\xi_8^n+\xi_9^n\},
\end{align*}
where $\xi_1^n$ is defined in Lemma \ref{apos:for:q:space}, $\xi_i^n,\;\;i=2,3,4,5$ are defined in Lemma \ref{inter:lemma:time:apos} and $\xi_i^n,\;\;i=6,\ldots,9$ are defined in Lemma \ref{inter:lemma:phi:apos}.
\end{theorem}
\begin{proof}
The proof begins with the use of triangle inequalities (\ref{Y:n}) and (\ref{phi:n}), as well as Theorem \ref{stability} and Proposition \ref{psi:stability}. The rest of the proof is completed by using Lemmas \ref{apos:for:q:space}-\ref{inter:lemma:phi:apos}.
\end{proof}
 \begin{remark}
 The estimators presented in Theorem \ref{main:thm:space} are contributed by the approximation errors of the control, state, and co-state variables.
 The estimators are generally influenced
by the approximation errors for the state and co-state, whereas $\xi_1^n$ is mostly determined by the approximation error for the control variable.
The co-state equation contributes the estimators $\xi_6^n$, $\xi_8^n$, $\xi_9^n$, while the state equation contributes the estimators $\xi_2^n, \ldots, \xi_5^n$.
These estimators are divided into three parts: the estimators $\xi_3^n$ and $\xi_9^n$ are generated by the approximation of time, the estimators $\xi_2^n$, $\xi_6^n$ are caused by the discretization of space and the estimators $\xi_4^n$ and $\xi_8^n$ are induced by the  data approximation.
For directing an adaptive algorithm, these estimators are extremely useful. 

\end{remark}

\begin{remark}
We choose $\tau=\delta_{t^*}$, where $\delta_{t^*}$ stands for the Dirac measure focused at time $t=t_{*}$.
Consider the set of indices $\mathcal{I}$ for the time partitions where the emphasis of the measure data $\delta_{t^{*}}$ is placed. Let $t_*\in (t_{n-1},t_{n}]$ for some $n\in\mathbb{N}$.
Then,  $\xi_5^n$ of Theorem \ref{main:thm:space} reduces to
\begin{align*}
\xi_5^n:=\sum_{K\in \mathcal{T}_{h,\mathcal{I}}}\Big(h_K^2\|\sigma(\cdot,t_*)\|^2_{L^2(K)}+k_{\mathcal{I}}\|\sigma(\cdot,t_*)\|_{L^2(K)}^2\Big).
\end{align*}

\end{remark}
\section{Numerical experiments}
The numerical results for a two-dimensional  problem are presented in this section to support our theoretical conclusions. The projection gradient method  is used to solve the optimization problem. The numerical tests are performed by  utilizing  the software
 Free Fem++ \cite{hetch} and all the constants are taken to be one. If $E$ is an error functional, then we define the order of convergence between two mesh of sizes $h_1>h_2$ as 
\begin{align*}
order = \frac{\log (E(h_1)/E(h_2))}{\log (h_1/ h_2)}.
\end{align*}

\textbf{ Data of the problem and true solution.} We use the example considered in \cite{pshakya19} on L-shape domain $\Omega=(0,1)^2\setminus [\frac{1}{2},1]^2$. We have used the point-wise control constraints  $u_a=-0.5$ and $u_b=0.1$. The final time $T=1$, and the remaining data of the problem are given as follows

\noindent
\begin{align*}
\sigma\tau&=\sin(\pi(x_1^2+x_2^2))\delta_{\frac{1}{2}}-u+\sin(\pi(x_1^2+x_2^2))\hat{\gamma}(t)+\{-4\pi
 \cos(\pi(x_1^2+x_2^2))\\&\quad\quad+4\pi^2(x_1^2+x_2^2) \sin(\pi(x_1^2+x_2^2))\}\tilde{\tilde{\gamma}}(t)
\end{align*}
with 
$
\hat{\gamma}(t)=\begin{cases} 2t\;\;\;& t<\frac{1}{2},\\
2t+2\;\;\;\; & t\geq \frac{1}{2},
\end{cases}
$
\hspace{0.9cm}
and
\hspace{0.9cm}
$\tilde{\tilde{\gamma}}(t)=\begin{cases}t^2\;\;\;\;\;\;\;\;\;\;\;\;t<\frac{1}{2},\\t^2+2t\;\;\;\;t\geq \frac{1}{2}.
\end{cases}$

\noindent
The desired state is
\begin{align*}
y_d=\sin(\pi(x_1^2+x_2^2))+4\pi t \cos(\pi(x_1^2+x_2^2))+(\tilde{\tilde{\gamma}}(t) -4\,\pi^2 t(x_1^2+x_2^2)) \sin(\pi(x_1^2+x_2^2)).
\end{align*}
 The exact  state and co-state  are given by
\begin{align*}
y&=\sin(\pi(x_1^2+x_2^2))\tilde{\tilde{\gamma}}(t),\;\;\;\;
\phi=\sin(\pi(x_1^2+x_2^2)) t,
\end{align*}
and the exact control is calculated by the formula (\ref{exp:for:qbar:time}) and (\ref{pointwise:projection}). 
 
\noindent
First, we validate the results obtained in Section $4$.  For the spatial discretization of the state and co-state variables the continuous piecewise linear polynomials   are utilized, whereas the piecewise constant functions are used for the control variable.
 With a uniform time step size of $k \approx h^{2s}$ with $s=0.6$, the backward Euler scheme is employed to approximate the time derivative. Table \ref{tab11} displays the order of convergence for various degrees of freedom (Dof) together with the  errors computed in the $L^2(0,T;L^2(\Omega))$-norm at final time $T=1$.  We observe that near the L-shape corner, the regularity of the state, co-state and control variables is not enough to get the linear rate of convergence.  
Table \ref{tab11} demonstrates that the error decreases as the Dof rises and we obtained the convergence rate matches with the results obtained in Section $4$.
The exact and discrete control profiles are depicted in Figure \ref{control}.

Next, we verify the findings from Section 5 of our studies. 
The error estimators derived in Section 5, are utilized as the error indicator in the adaptive loop.
 \begin{align*}
\text{SOLVE}\rightarrow \text{ESTIMATE}\rightarrow \text{MARK}\rightarrow \text{REFINE}
\end{align*}
The development of the space-time algorithm  is based on \cite{manohar21}. 
 To see the performance of  a posteriori error estimators, we set the time step size $k_n \approx h_K^{2s}$ with $s=0.6$. The space and time tolerances are taken to be $10^{-2}$.
 Table \ref{tab:1} shows the error and convergence rate for the state, co-state and control variables in the adaptive meshes. It has been noted that the local refinement of the meshes improves the convergence rate. In Figure \ref{adaptive:msh}, we present the adaptive meshes at different level of refinements.     Figure \ref{adaptive:msh} demonstrates how effectively the mesh adapts  in the vicinity of the L-shape corner and a large number of nodes is distributed  along this corner.

 \begin{figure}[h!]
\centering
\begin{subfigure}{.5\textwidth}
  \centering
  \includegraphics[width=8cm, height=5cm]{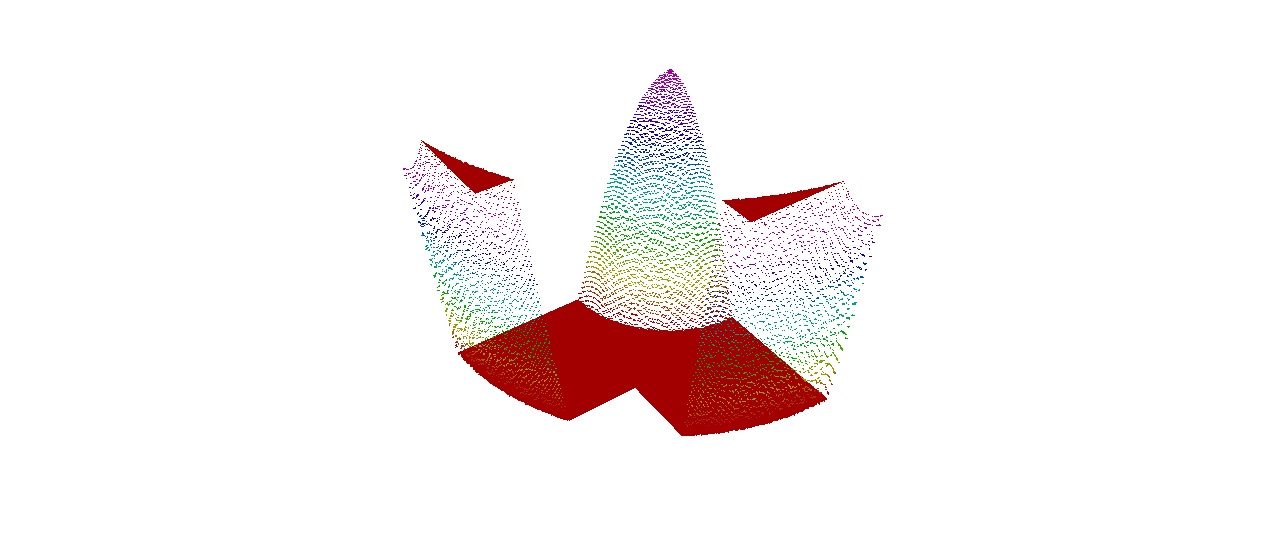}
  \caption{Approximate control}
\end{subfigure}%
\begin{subfigure}{.5\textwidth}
 \centering
 \includegraphics[width=8cm, height=5cm]{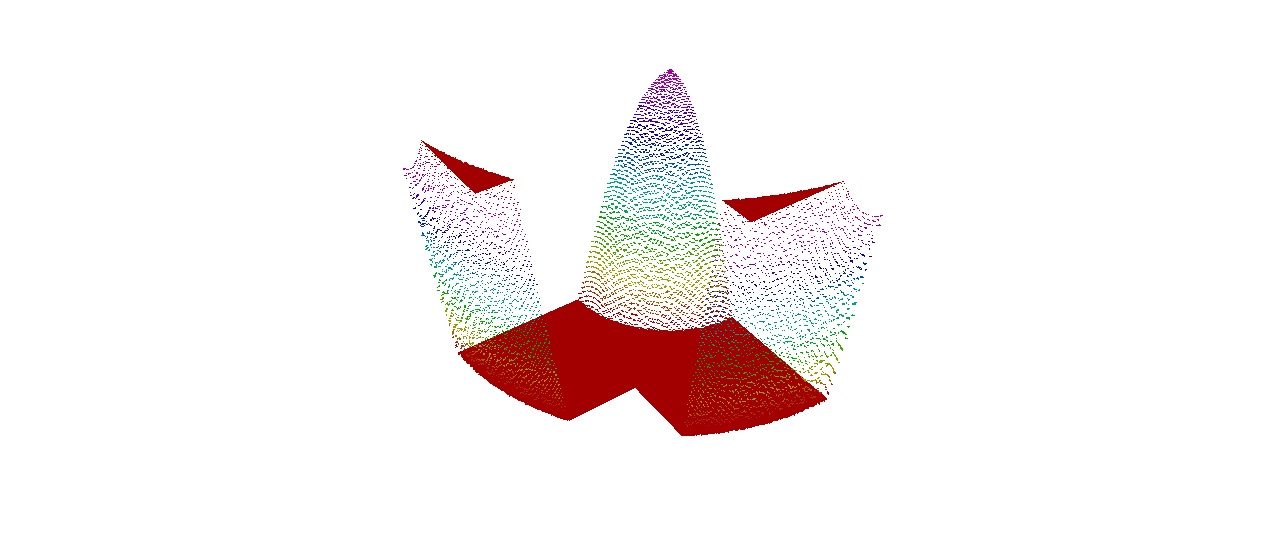}
   \caption{ Exact control}
\end{subfigure}
\caption{The profiles  of the control variable  at $T=1$ with $11785$ Dof.}\label{control}
\end{figure}

\begin{small}
\begin{table}[h!]
\begin{center}
\caption{Errors of  the state $y$, co-state $\phi$ and control $u$ variables on uniform meshes}\label{tab11}
\begin{tabular}{|c|c|c|c|c|c|c|c|}
\hline
Dof & N&$\|y-Y_h\|_{L^2(0,T;L^2(\Omega))}$ &order&$\|\phi-\Phi_h\|_{L^2(0,T;L^2(\Omega))}$ &order & $\|u-U_h\|_{L^2(0,T;L^2(\Omega))}$&order \\
\hline
 25 &  4&$9.43224\times 10^{-1}$   &    -  & $ 3.52156\times 10^{-1} $ &      -    &$1.96823\times 10^{-1} $ &  -\\
 81 & 8&  $6.58948\times 10^{-1}$ &0.6101& $2.37896\times 10^{-1}$ & 0.6672  &$1.32783\times 10^{-1}$  &  0.6695\\
289  &  16&$4.39783\times 10^{-1}$  &0.6359& $1.57783\times 10^{-1}$ & 0.6457   & $8.75573\times 10^{-2}$ &  0.6549\\
1089 &32 &$2.82896\times 10^{-1}$  &0.6652& $9.98987\times 10^{-2}$& 0.6891 & $5.64994\times 10^{-2}$ & 0.6604   \\
4225& 64& $1.79678\times 10^{-1}$&0.6696&$6.37896\times 10^{-2}$& 0.6617& $3.62287\times 10^{-2}$ & 0.6555\\
\hline
\end{tabular}
\end{center}
\end{table} 
\end{small}

\begin{small}
\begin{table}[h!]
\begin{center}
\caption{Errors of  the state $y$, co-state $\phi$ and control $u$ variables on  adaptive meshes}\label{tab:1}
\begin{tabular}{|c|c|c|c|c|c|c|c|}
\hline
Dof & N&$\|y-Y_h\|_{L^2(0,T;L^2(\Omega))}$ &order&$\|\phi-\Phi_h\|_{L^2(0,T;L^2(\Omega))}$ &order & $\|u-U_h\|_{L^2(0,T;L^2(\Omega))}$&order \\
\hline
 36 &  4&$8.26735\times 10^{-1}$   &    -  & $ 1.39846\times 10^{-1} $ &      -    &$9.35682\times 10^{-1} $ &  -\\
209 & 8&  $3.29768\times 10^{-1}$ &1.0451& $5.95639\times 10^{-2}$ & 0.9705 &$4.17895\times 10^{-1}$  &  0.9166\\
417  &  16&$2.25869\times 10^{-1}$  &1.0957& $4.19460\times 10^{-2}$ & 1.0153   & $2.96932\times 10^{-1}$ &  0.9894\\
615& 32 &$1.86984\times 10^{-1}$  &0.9725& $3.39783\times 10^{-2}$& 1.0844 & $2.43673\times 10^{-1}$ & 1.0175  \\
1445& 64& $1.19182\times 10^{-1}$&1.0570&$2.15249\times 10^{-2}$& 1.0714& $1.56695\times 10^{-1}$ & 1.0363\\
\hline
\end{tabular}
\end{center}
\end{table} 
\end{small}
\begin{figure}[h!]
\centering
\begin{subfigure}{.4\textwidth}
 \centering
 \includegraphics[width=3.5cm, height=3.5cm]{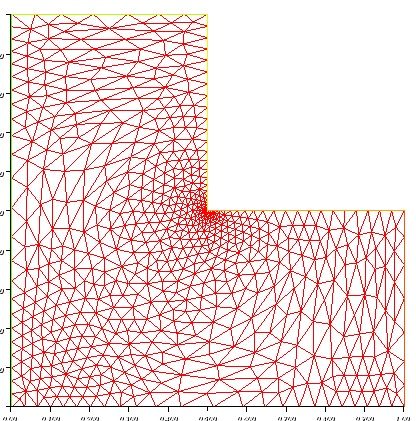}
   \caption{Step-1}
\end{subfigure}
\begin{subfigure}{.4\textwidth}
 \centering
 \includegraphics[width=3.5cm, height=3.5cm]{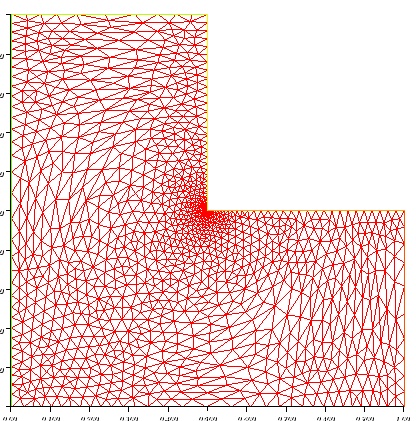}
   \caption{Step-2}
\end{subfigure}
\begin{subfigure}{.4\textwidth}
 \centering
 \includegraphics[width=3.5cm, height=3.5cm]{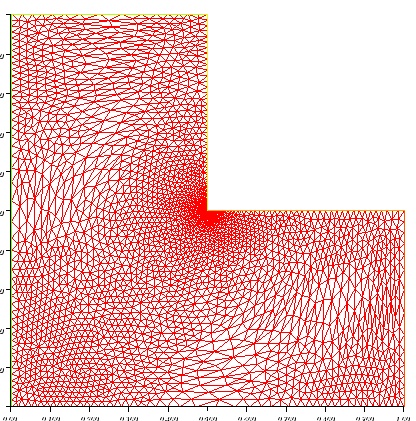}
   \caption{Step-3}
\end{subfigure}
\begin{subfigure}{.4\textwidth}
 \centering
 \includegraphics[width=3.5cm, height=3.5cm]{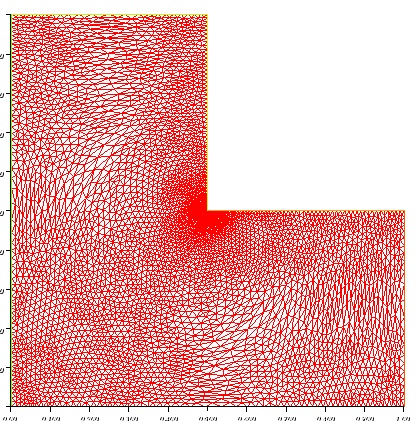}
   \caption{Step-4}
\end{subfigure}
\caption{Adaptive meshes for the state at different level of refinements  at time $T=1$.}\label{adaptive:msh}
\end{figure}

\noindent
\textbf{Acknowledgements.} 
The author wishes to thank the anonymous referees for their valuable comments and constructive suggestions that lead to the improvement of the content of the manuscript.


 \end{document}